\documentclass[a4paper,12pt]{amsart}

\usepackage{amsmath, amssymb, amscd, amsthm, verbatim,enumerate,a4wide}
\usepackage{hyperref}
\usepackage{tikz}
\usetikzlibrary{positioning,trees,decorations.pathreplacing,calc}
\tikzset{
    Solid/.style={circle,draw,inner sep=1.5,fill=black},
    level distance=.6cm,
    sibling distance=.6cm,
    font=\footnotesize
}

\newtheorem{thm}{Theorem}[section]
\newtheorem{cor}[thm]{Corollary}
\newtheorem{lemma}[thm]{Lemma}
\newtheorem{prop}[thm]{Proposition}

\theoremstyle{definition}
\newtheorem{remark}[thm]{Remark}
\numberwithin{equation}{section}

\newcommand{\N}{\mathbb N}
\newcommand{\C}{\mathbb C}
\newcommand{\Z}{\mathbb Z}

\newcommand{\al}{\alpha}
\newcommand{\be}{\beta}
\newcommand{\ga}{\gamma}

\newcommand{\de}{\delta}
\newcommand{\De}{\Delta}

\newcommand{\si}{\sigma}

\newcommand{\la}{\lambda}
\newcommand{\La}{\Lambda}

\newcommand{\bn}{\mathbf{n}}
\newcommand{\br}{\mathbf{r}}
\newcommand{\bs}{\mathbf{s}}
\newcommand{\bt}{\mathbf{t}}
\newcommand{\bp}{\mathbf{p}}
\newcommand{\bv}{\mathbf{v}}
\newcommand{\bnu}{\boldsymbol{\nu}}

\newcommand{\tensor}{\otimes}
\newcommand{\rphis}[5]{\,_{#1}\varphi_{#2}\!\left( \genfrac{.}{.}{0pt}{}{#3}{#4}
\,;#5 \right)}

\newcommand{\mhyphen}{\text{--}}

\begin{document}
\title[$3nj$-symbols and $q$-Bessel functions]{$3nj$-symbols and identities for $q$-Bessel functions}
\author{Wolter Groenevelt}
\address{Technische Universiteit Delft, DIAM, PO Box 5031,
2600 GA Delft, the Netherlands}
\email{w.g.m.groenevelt@tudelft.nl}

\maketitle

\begin{abstract}
The $6j$-symbols for representations of the $q$-deformed algebra of polynomials on $\mathrm{SU}(2)$ are given by Jackson's third $q$-Bessel functions. This interpretation leads to several summation identities for the $q$-Bessel functions. Multivariate $q$-Bessel functions are defined, which are shown to be limit cases of multivariate Askey-Wilson polynomials. The multivariate $q$-Bessel functions occur as $3nj$-symbols.
\end{abstract}

\section{Introduction}
It is well known that Wigner's $6j$-symbols for the $\mathrm{SU}(2)$ group are multiples of hypergeometric orthogonal polynomials called the Racah polynomials. Similarly, $6j$-symbols for the $\mathrm{SU}(2)$ quantum group can be expressed in terms of $q$-Racah polynomials, which are $q$-hypergeometric orthogonal polynomials. With this interpretation, properties of $6j$-symbols such as summation formulas and orthogonality relations, lead to properties of specific families of orthogonal polynomials.

In this paper we consider $6j$-symbols for representations of the the $q$-deformed algebra of polynomials on $\mathrm{SU}(2)$. This algebra has as irreducible representations the trivial one, and a family of infinite dimensional representations which dissappear in the classical limit. The $6j$-symbols for tensor products of the three infinite dimensional representations can be expressed in terms of Jackson's third $q$-Bessel functions \cite{Groen14}. Note that, different from the classical $6j$-symbols, these are not polynomials. We consider three fundamental identities for $6j$-symbols (see e.g.~\cite{CFS}): Racah's backcoupling identity, the Biedenharn-Elliott identity and the hexagon identity. These identities are obtained by decomposing 3- or 4-fold tensor product representations in several ways. To keep track of the order of decomposing the representations, it is convenient to identify certain vectors in the representations spaces with binary trees. Then the $6j$-symbols can be considered as coupling coefficients between two of these trees. The identities we obtain can be interpreted as summation identities for $q$-Bessel functions. We remark that the hexagon identity implies that the $q$-Bessel functions are matrix elements of an infinite dimensional solution of the quantum Yang-Baxter equation (or, the star-triangle equation in IRF-models), see e.g.~\cite{Ji}, which should be of independent interest.

We also consider specific $3nj$-symbols, which may naturally be considered as multivariate $q$-Bessel functions. The one variable $q$-Bessel functions fit into an extended Askey-scheme \cite{KoeSt} of orthogonal $q$-hypergeometric functions; the original ($q$-)Askey-scheme \cite{KLS} consists of ($q$-)hypergeometric orthogonal polynomials. We will show that the multivariate $q$-Bessel functions fit into an extended Askey-scheme of multivariate orthogonal functions of $q$-hypergeometric type, by showing that the multivariate $q$-Bessel functions can be obtained as limits of the multivariate Askey-Wilson polynomials defined by Gasper and Rahman \cite{GR05}, which are the $q$-analogs of Tratnik's multivariate Wilson polynomials \cite{Tr}. The multivariate Askey-Wilson polynomials can be thought of as being on top of a scheme of multivariate orthogonal polynomials; several limit cases are considered in \cite{GR05}, \cite{GR07}, \cite{Il}. Geronimo and Illiev \cite{GeIl} obtained multivariate Askey-Wilson functions generalizing the multivariate Askey-Wilson functions, which should be on top of the extended Askey-scheme. Several families of orthogonal polynomials in this scheme and its $q=1$ analog are connected to tensor product representations and binary coupling schemes, see e.g.~Van der Jeugt \cite{Jeugt}, Rosengren \cite{Ro} and Scarabotti \cite{Sc}.

This paper is organized as follows. In Section \ref{sec:quantumSU(2)} the quantum algebra $\mathcal A_q(\mathrm{SU}(2))$ and its representation theory is recalled. In section \ref{sec:6j-symbols} it is shown that the $6j$-symbols are essentially $q$-Bessel functions, using a generating function for $q$-Bessel functions. Using binary trees we obtain the fundamental identities for $6j$-symbols, leading to summation formulas for the $q$-Bessel functions. In Section \ref{sec:3nj-symbols} we first define multivariate $q$-Bessel functions as nontrivial products of $q$-Bessel functions, and we prove orthogonality relations. Then we show that these multivariate $q$-Bessel functions occur as $3nj$-symbols, and use this interpretation to find a summation formula.\\

\textit{Notations.} We use $\N = \{0,1,2,\ldots\}$ and we use standard notation for $q$-hypergeometric functions as in \cite{GR}.

\section{The quantum algebra $\mathcal A_q(\mathrm{SU}(2))$} \label{sec:quantumSU(2)}
Let $q \in (0,1)$. The $q$-deformed algebra of polynomials on ${\mathrm{SU}(2)}$ is the complex unital associative algebra $\mathcal A_q=\mathcal A_q(\mathrm{SU}(2))$  generated by $\al$, $\be$, $\ga$, $\de$, which satisfy the relations
\begin{equation} \label{eq:comm rel}
\begin{gathered}
\al \be = q \be \al, \quad \al \ga = q \ga \al, \quad \be \de = q \de \be, \quad \ga \de = q \de \ga, \\
\be \ga = \ga \be, \quad \al \de - q\be \ga = 1 = \de \al - q^{-1}\be \ga.
\end{gathered}
\end{equation}
$\mathcal A_q$ is a Hopf-$*$-algebra with $*$-structure and comultiplication $\De$ defined on the generators by
\begin{equation} \label{eq:*structure}
\al^* = \de, \quad \be^* = -q \ga, \quad \ga^* = -q^{-1} \be, \quad \de^* = \al,
\end{equation}
\begin{equation} \label{eq:coprod}
\begin{split}
\De(\al) = \al \otimes \al + \be \otimes \ga, \quad \De(\be) = \al \otimes \be + \be \otimes \de,\\
\De(\ga) = \ga \otimes \al + \de \otimes \ga, \quad \De(\de) = \de \otimes \de + \ga \otimes \be.
\end{split}
\end{equation}
The irreducible $*$-representation of $\mathcal A_q$ are either 1-dimensional, or infinite dimensional. The infinite dimensional irreducible $*$-representations are labeled by $\phi \in [0,2\pi)$, and we denote a representation by $\pi_\phi$. The representation space of $\pi_\phi$ is $\ell^2(\N)$. The generators $\al,\be,\ga,\de$ act on the standard orthonormal basis $\{e_n \mid n \in \N\}$ of $\ell^2(\N)$ by
\[
\begin{split}
\pi_\phi(\al)\, e_n &= \sqrt{1-q^{2n}}\, e_{n-1},\\
\pi_\phi(\be)\,e_n &= - e^{-i\phi} q^{n+1}\, e_n,\\
\pi_\phi(\ga)\, e_n &= e^{i\phi} q^n\, e_n,\\
\pi_\phi(\de)\, e_n &= \sqrt{1-q^{2n+2}} e_{n+1}.
\end{split}
\]
Note that $\pi_\phi(\ga\be)$ is a self-adjoint diagonal operator in the standard basis.

\begin{remark}
In this paper we consider tensor products of $\pi_0$. We could also consider the representation $\pi_{\phi_1} \tensor \pi_{\phi_2}$, but this would not lead to more general results in this paper, because representation labels only occur in phase factors, see \cite[\S II.A]{Groen14}.
\end{remark}

Let $\si: \ell^2(\N) \tensor \ell^2(\N) \to \ell^2(\N) \tensor \ell^2(\N)$ be the flip operator, the linear operator defined on pure tensors by $\si(v_1 \tensor v_2) = v_2 \tensor v_1$. We write
\[
\pi_{12} = (\pi_0 \tensor \pi_0) \De, \qquad \pi_{21}= \si \pi_{12} \si.
\]
For three-fold tensor product representations we write
\[
\pi_{1(23)} = (\pi_0 \tensor \pi_0 \tensor \pi_0)(1\tensor \De)(\De), \qquad \pi_{(12)3} = (\pi_0\tensor \pi_0 \tensor \pi_0)(\De\tensor 1)(\De).
\]
Since $\De$ is coassociative, we have $\pi_{1(23)} = \pi_{(12)3}$.

From \eqref{eq:coprod} one finds
\[
\De(\ga \ga^*) =q^{-1}\De(\ga\be)= -q^{-1}\Big( \ga \be \tensor \al \de+ \ga \al \tensor \al \be + \de \be \tensor \ga \de + \de \al \tensor \ga \be\Big).
\]
Using this, eigenvectors of $\pi_{12}(\ga \ga^*)$ can be computed (see \cite{Groen14} for details):
for $p \in \Z$ and $x \in \N$ define
\[
e_{x,p}^{12} =
\sum_{\substack{n,m \in \N\\n-m=p}} C_{x,m,n} \, e_m \tensor e_{n},
\]
where we assume $e_{-n}=0$ for $n \geq 1$,
then $\pi_{12}(\ga \ga^*) e_{x,p}^{12} = q^{2x} e_{x,p}^{12}$. The Clebsch-Gordan coefficients $C_{x,m,n}$ can be given explicitly in terms of Wall polynomials, see \cite{KLS}, which are defined by
\begin{equation} \label{eq:defWallpol}
\begin{split}
p_n(q^x;a;q) &= \rphis{2}{1}{ q^{-n}, 0 }{aq}{q, q^{x+1}}\\
& = \frac{ (-a)^n q^{\frac12n(n+1)}}{(aq;q)_n} \rphis{2}{0}{q^{-n},q^{-x}}{\mhyphen}{q,\frac{q^x}{a}},
\end{split}
\end{equation}
for $n,x \in \N$.
The second expression follows from applying transformation \cite[(III.8)]{GR} with $b \to 0$. Note that, for $x\in \N$, the $_2\varphi_0$-series can be considered as a polynomial in $q^{-n}$ of degree $x$. This polynomial is (proportional to) an Al-Salam--Carlitz II polynomial.

Let the function $\bar p_n(q^x;a;q)$ be defined by
\begin{equation} \label{eq:defbarpn}
\bar p_n(q^x;a;q)=(-1)^{n+x} \sqrt{ \frac{ (aq)^{x-n} (aq;q)_\infty (aq;q)_n }{ (q;q)_n (q;q)_x } }\, p_n(q^x;a;q),
\end{equation}
then from the orthogonality relation for the Wall polynomials and from completeness we obtain the orthogonality relations
\[
\sum_{x\in \N} \bar p_n(q^x;a;q) \bar p_m(q^x;a;q) = \de_{nm}, \qquad \sum_{n \in \N} \bar p_n(q^x;a;q) \bar p_n(q^y;a;q) = \de_{xy},
\]
for $0<a<q^{-1}$. The second relation corresponds to orthogonality relations for Al-Salam--Carlitz II polynomials. The coefficients $C_{x,m,n}$, $m,n \in \N$, are defined by
\[
C_{x,m,n} =
\begin{cases}
\bar p_n(q^{2x};q^{2(n-m)};q^2), & n \geq m,\\
\bar p_m(q^{2x};q^{2(m-n)};q^2) , & n \leq m,
\end{cases}
\]
and they satisfy
\begin{equation} \label{eq:c p<->-p}
C_{x,n,m} = C_{x,m,n},
\end{equation}
which follows from the explicit expression as a $_2\varphi_1$-function. Furthermore, we define $C_{x,m,n}=0$ for $m \in -\N_{\geq 1}$ or $n \in -\N_{\geq 1}$ or $x \in -\N_{\geq 1}$.

The set $\{e_{x,p}^{12} \mid p \in \Z, x \in \N \}$ is an orthonormal basis for $\ell^2(\N) \tensor \ell^2(\N)$. The actions of the $\mathcal A_q$-generators on this basis are given by
\begin{equation} \label{eq:actions1}
\begin{split}
\pi_{12}(\al)\, e_{x,p}^{12} &= \sqrt{ 1- q^{2x}}\, e_{x-1,p}^{12} ,\\
\pi_{12}(\be)\, e_{x,p}^{12} &= -q^{x+1}\, e_{x,p+1}^{12},\\
\pi_{12}(\ga)\, e_{x,p}^{12} &= q^x\,e_{x,p-1}^{12},\\
\pi_{12}(\de)\, e_{x,p}^{12} &= \sqrt{ 1- q^{2x+2}}\, e_{x+1,p}^{12},
\end{split}
\end{equation}
where $e_{-1,p}^{12}=0$. We can also find eigenvectors $e_{x,p}^{21}$ of $\pi_{21}(\ga \ga^*)$ for eigenvalue $q^{2x}$, $x \in \N$:
\[
e_{x,p}^{21} = \sum_{\substack{n,m \in \N\\m-n=p}} C_{x,m,n}\, e_m \tensor e_n = e^{12}_{x,-p}.
\]

\section{$6j$-symbols and $q$-Bessel functions} \label{sec:6j-symbols}
In \cite{Groen14} explicit expressions for the $6j$-symbols (and for more general coupling coefficients) have been found. It turns out that they are essentially $q$-Bessel functions. Here we derive these results again using a more direct approach, and use this interpretation of the $q$-Bessel functions to obtain summation identities.

\subsection{$6j$-symbols}
In the same way as above we can find eigenvectors of $\pi_{1(23)}(\ga\ga^*)$ and $\pi_{(12)3}(\ga\ga^*)$; for $x\in \N$, $p,r \in \Z$,
\begin{align*}
e_{x,p,r}^{1(23)} &= \sum_{n \in \N} C_{x,n,n+p} \ e_n \tensor e_{n+p,x-n-r}^{23} &\\
&= \sum_{n,m \in \N} C_{x,n,n+p} C_{n+p,m,k}\, e_n \tensor e_m \tensor e_k, &n-m+k=x-r,\\
e_{x,p,r}^{(12)3} &= \sum_{k \in \N} C_{x,k-p,k} \ e_{k-p,r-x+k}^{12} \tensor e_k & \\
&= \sum_{k,m \in \N} C_{x,k-p,k} C_{k-p,n,m}\, e_n \tensor e_m \tensor e_k,& \qquad n-m+k=x-r,
\end{align*}
are eigenvectors for eigenvalue $q^{2x}$, $x \in \N$. We use here the convention $e_{-n}=e_{-n,p} = 0$ for $n \in -\N_{\geq 1}$. The actions of the $\mathcal A_q$-generators $\al, \be, \ga, \de$ on the eigenvectors can be obtained in the same way as in \cite{Groen14}
\begin{align*}
\pi_{1(23)}(\al) e_{x,p,r}^{1(23)} &= \sqrt{ 1-q^{2x} }\, e_{x-1,p,r}^{1(23)},& \mathcal \pi_{(12)3}(\al) e_{x,p,r}^{(12)3} &= \sqrt{ 1-q^{2x} }\, e_{x-1,p,r}^{(12)3}, \\
\pi_{1(23)}(\be) e_{x,p,r}^{1(23)} &= -q^{x+1}\, e_{x,p+1,r}^{1(23)},& \mathcal \pi_{(12)3}(\be) e_{x,p,r}^{(12)3} &= -q^{x+1}\, e_{x,p+1,r}^{(12)3},\\
\pi_{1(23)}(\ga) e_{x,p,r}^{1(23)} &= q^x\, e_{x,p-1,r}^{1(23)},& \pi_{(12)3}(\ga) e_{x,p,r}^{(12)3} &= q^x\, e_{x,p-1,r}^{(12)3}\\
\pi_{1(23)}(\de) e_{x,p,r}^{1(23)} &= \sqrt{ 1-q^{2x+2} }\, e_{x+1,p,r}^{1(23)},& \pi_{(12)3}(\de) e_{x,p,r}^{(12)3} &= \sqrt{ 1-q^{2x+2} }\, e_{x+1,p,r}^{(12)3},
\end{align*}
where $e_{-1,p,r}=0$.
Note that this corresponds exactly to the actions on the eigenvectors $e_{x,p}$.\\

The $6j$-symbol (or Racah coefficient) $R_{p_1,r_1;p_2,r_2}^{x}$ is the (re)coupling coefficient between the two eigenvectors;
\[
R_{p_1,r_1;p_2,r_2}^{x} = \langle e_{x,p_1,r_1}^{1(23)}, e_{x,p_2,r_2}^{(12)3} \rangle,
\]
or equivalently
\begin{equation} \label{eq:e=sum Re}
e_{x,p_1,r_1}^{1(23)} = \sum_{p_2,r_2} R_{p_1,r_1;p_2,r_2}^x  e_{x,p_1,r_1}^{1(23)}.
\end{equation}
We start by looking at some simple properties of $R$.
\begin{prop} \label{prop:properties R}
The coefficients $R$ have the following properties:
\begin{enumerate}[(i)]
\item Orthogonality relations: $\displaystyle \sum_{p_1,r_1 \in \Z} R_{p_1,r_1;p_2,r_2}^{x} R_{p_1,r_1;p_3,r_3}^{x} = \de_{p_2,p_3} \de_{r_2,r_3}$.
\item $R_{p_1,r_1;p_2,r_2}^{x} = R_{p_1+k,r_1;p_2+k,r_2}^{x}$ for $k \in \Z$.
\item $R_{p_1,r_1;p_2,r_2}^{x} = R_{p_1,r_1;p_2,r_2}^{x+k}$ for $k \in \Z_{\geq -x}$.
\item For $k,m,n \in \N$,
\[
C_{x,n+p_1,n} C_{n+p_1,m,k} = \sum_{p_2 \in \Z_{\leq k}} R_{p_1,r;p_2,r}^x C_{x,k-p_2,k} C_{k-p_2,m,n}, \qquad x-r=n-m+k.
\]
\item Duality: $R_{p_1,r;p_2,r}^x = R_{-p_2,r;-p_1,r}^x$.
\end{enumerate}
\end{prop}
Note that identity (iii) implies that $R$ is independent of $x$, therefore we will omit the superscript `$x$'.
\begin{proof}
The coefficients $R$ are matrix coefficients of a unitary operator, which leads to the orthogonality relations. The next two identities follow from the $*$-structure of $\mathcal A_q$. From $\be^* = -q\ga$ we obtain
\[
\langle e_{x,p_1\pm 1,r_1}^{1(23)}, e_{x,p_2,r_2}^{(12)3} \rangle = \langle e_{x,p_1,r_1}^{1(23)}, e_{x,p_2 \mp 1,r_2}^{(12)3} \rangle,
\]
which implies (ii). Identity (iii) follows from $\al^*=\de$.
Identity (iv) follow from the expansion
\[
e_{x,p_1,r_1}^{1(23)} =\sum_{p_2,r_2\in \Z} R_{p_1,r_1;p_2,r_2}\, e_{x,p_2,r_2}^{(12)3}.
\]
by taking inner products with $e_{n} \tensor e_m \tensor e_k$. The duality property follows from identity (iv).
\end{proof}

\subsection{$q$-Bessel functions}
Define
\[
J_\nu(x;q)= x^\frac{\nu}{2} \frac{ (q^{\nu+1};q)_\infty}{(q;q)_\infty}\rphis{1}{1}{0}{q^{\nu+1}}{q,qx}, \qquad x\geq 0,\ \nu \not\in \Z_{\leq -1},
\]
which is Jackson's third $q$-Bessel function (also known as the Hahn-Exton $q$-Bessel function), see e.g.~\cite{KooSw}. We will need the $q$-Bessel functions with $\nu \in \Z$; for negative integers they are defined by 
\[
J_{-n}(x;q) = (-1)^n q^{\frac{n}{2}} J_n(xq^n;q), \qquad n \in \N,
\]
see \cite[(2.6)]{KooSw}. We will use the following generating function to identify the $6j$-symbols with $q$-Bessel functions.
\begin{prop} \label{prop:generating function}
For $|t|< 1$,
\[
\sum_{m=0}^\infty q^{-\frac{\nu m}{2}}J_{\nu}(xq^{m};q) \frac{ t^m }{(q;q)_m} = x^\frac{\nu}{2} \frac{(q^{\nu+1};q)_\infty}{(q,t;q)_\infty } \rphis{1}{1}{t}{q^{\nu+1}}{q,qx}.
\]
\end{prop}
\begin{proof}
Write $J_\nu$ as a $_1\varphi_1$-series, interchange the order of summation, and use summation formula \cite[(II.1)]{GR};
\[
\begin{split}
\sum_{m=0}^\infty \rphis{1}{1}{0}{q^{\nu+1}}{q,xq^{m+1}} \frac{t^m }{(q;q)_m} &= \sum_{k=0}^\infty \frac{ (-1)^k q^{\frac12 k(k-1)} (xq)^k }{(q,q^{\nu+1};q)_k} \sum_{m=0}^\infty \frac{q^{mk} t^m}{(q;q)_m} \\
& = \sum_{k=0}^\infty \frac{ (-1)^k q^{\frac12 k(k-1)} (xq)^k }{(q,q^{\nu+1};q)_k (tq^{k};q)_\infty }. \qedhere
\end{split}
\]
\end{proof}
If $t^{-1}q^{\nu+1}\in q^{-\N}$ the right hand side in the Proposition \ref{prop:generating function} can be written in terms of a Wall polynomial, which gives the following special case.
\begin{cor} \label{cor:generating function}
For $n \in \N$,
\[
\sum_{m=0}^\infty q^{-\frac12(\nu-n)m}J_{\nu-n}(xq^{m};q) \frac{ q^{m(\nu+1)} }{(q;q)_m} = x^{\frac12(\nu-n)} \frac{(qx;q)_\infty}{(q;q)_\infty } p_n(q^{\nu};x;q).
\]
\end{cor}
\begin{proof}
In Proposition \ref{prop:generating function} replace $\nu$ by $\nu-n$, set $t=q^{\nu+1}$, and use the transformation
\[
\rphis{1}{1}{A}{B}{q,Z} = \frac{ (A,Z;q)_\infty }{ (B;q)_\infty } \rphis{2}{1}{ 0 , B/A}{Z}{q,A},
\]
(which is a special case of \cite[(III.4)]{GR}) and the definition \eqref{eq:defWallpol} of the Wall polynomials.
\end{proof}
We are now in a position to show that the $6j$-symbols are essentially $q$-Bessel functions.

\begin{prop} \label{prop:R=qBessel}
For $p_1,p_2,r_1,r_2\in \Z$,
\[
R_{p_1,r_1;p_2,r_2} = \de_{r_1,r_2}\,(-q)^{p_1-p_2} J_{r_1}(q^{2p_1-2p_2};q^2). \qedhere
\]
\end{prop}
\begin{proof}
We write out Proposition \ref{prop:properties R}(iv) for $m=k=0$, and we replace $p_2$ by $-p_2$,
\[
\sum_{p_2 \in \N} R_{p_1,r;-p_2,r} \frac{ (-1)^{p_2} q^{p_2(n+x+1)} }{(q^2;q^2)_{p_2}} = \frac{(-1)^{p_1} q^{p_1(x-n+1)}}{ (q^2;q^2)_{p_1} } p_n(q^{2x};q^{2p_1};q^2), \qquad x-r=n,
\]
then the result follows from Corollary \ref{cor:generating function}.
\end{proof}

\subsection{Identities}
Several classical identities for $6j$-symbols for $\mathrm{SU}(2)$ remain valid for our $6j$-symbols. By Proposition \ref{prop:R=qBessel} these can be interpreted as identities for $q$-Bessel functions. \\

First of all, the orthogonality relations for the $6j$-symbols from Proposition \ref{prop:properties R} are equivalent to the well-known $q$-Hankel orthogonality relations, see \cite[(2.11)]{KooSw}, for the $q$-Bessel functions $J_\nu$.
\begin{thm} \label{thm:orthogonality}
For $n,m \in \Z$,
\[
\sum_{x \in \Z} J_\nu(q^{x+m};q)J_\nu(q^{x+n};q) q^{x} = \de_{m,n} q^{-n}  .
\]
\end{thm}

To derive other identities it is convenient to represent eigenvectors of $\ga\ga^*$ as binary trees, see e.g.~Van der Jeugt's lecture notes \cite{Jeugt} for more details. We denote
\[
e_{x,n_2-n_1}^{12} = \
\begin{tikzpicture}
[baseline={([yshift=-.5ex]current bounding box.center)}]
\node[Solid](x){}
    child{node[Solid](n1){}}
    child{node[Solid](n2){}};
\node[above] at (x) {$x$};
\node[below] at (n1) {$n_1$};
\node[below] at (n2) {$n_2$};
\end{tikzpicture}
\]
where $n_1,n_2,x \in \Z$. Equivalently, we can identify this tree with the Clebsch-Gordan coefficient $C_{x,n_1,n_2}$, similar as in \cite{Sc}.
The identity $e_{x,p}^{12} = e_{x,-p}^{21}$, which is equivalent to \eqref{eq:c p<->-p}, is represented as
\begin{equation} \label{eq:tree flip}
\begin{tikzpicture}
[baseline={([yshift=-.5ex]current bounding box.center)}]
\node[Solid](x){}
    child{node[Solid](n1){}}
    child{node[Solid](n2){}};
\node[above] at (x) {$x$};
\node[below] at (n1) {$n_1$};
\node[below] at (n2) {$n_2$};
\end{tikzpicture}
\
=\
\begin{tikzpicture}
[baseline={([yshift=-.5ex]current bounding box.center)}]
\node[Solid](x){}
    child{node[Solid](n1){}}
    child{node[Solid](n2){}};
\node[above] at (x) {$x$};
\node[below] at (n1) {$n_2$};
\node[below] at (n2) {$n_1$};
\end{tikzpicture}
\end{equation}
where $p=n_1-n_2$. By coupling two of these we can represent eigenvectors corresponding to  three-fold tensor products:
\[
e_{x,p_1,r_{123}}^{1(23)} =
\begin{tikzpicture}
[baseline={([yshift=-.5ex]current bounding box.center)}]
\node[Solid](x){}
    child{
        child{node[Solid](n1){}}
        child[missing]
        }
    child{node[Solid](p1){}
        child{node[Solid](n2){}}
        child{node[Solid](n3){}}
        };
\node[above] at (x) {$x$};
\node[below] at (n1) {$n_1$};
\node[below] at (n2) {$n_2$};
\node[below] at (n3) {$n_3$};
\node[right] at (p1) {$p_1'$};
\end{tikzpicture}
\qquad \qquad
e_{x,p_2,r_{123}}^{(12)3} =
\begin{tikzpicture}
[baseline={([yshift=-.5ex]current bounding box.center)}]
\node[Solid](x){}
    child{node[Solid](p2){}
        child{node[Solid](n1){}}
        child{node[Solid](n2){}}
        }
    child{
        child[missing]
        child{node[Solid](n3){}}
        };
\node[above] at (x) {$x$};
\node[below] at (n1) {$n_1$};
\node[below] at (n2) {$n_2$};
\node[below] at (n3) {$n_3$};
\node[left] at (p2) {$p_2'$};
\end{tikzpicture}
\]
where $p_1'=n_1+p_1$, $p_2'=n_3-p_2$ and $r_{ijk} = x-n_i+n_j-n_k$ for $i,j,k \in \{1,2,3\}$. Now we can e.g.~represent the identities $e_{x,p_1,r_{123}}^{1(23)} = e_{x,p_1,r_{132}}^{1(32)} = e_{x,-p_1,r_{231}}^{(23)1}$ by
\[
\begin{tikzpicture}
[baseline={([yshift=-.5ex]current bounding box.center)}]
\node[Solid](x){}
    child{
        child{node[Solid](n1){}}
        child[missing]
        }
    child{node[Solid](p1){}
        child{node[Solid](n2){}}
        child{node[Solid](n3){}}
        };
\node[above] at (x) {$x$};
\node[below] at (n1) {$n_1$};
\node[below] at (n2) {$n_2$};
\node[below] at (n3) {$n_3$};
\node[right] at (p1) {$p_1'$};
\end{tikzpicture}
\quad  = \quad
\begin{tikzpicture}
[baseline={([yshift=-.5ex]current bounding box.center)}]
\node[Solid](x){}
    child{
        child{node[Solid](n1){}}
        child[missing]
        }
    child{node[Solid](p1){}
        child{node[Solid](n2){}}
        child{node[Solid](n3){}}
        };
\node[above] at (x) {$x$};
\node[below] at (n1) {$n_1$};
\node[below] at (n2) {$n_3$};
\node[below] at (n3) {$n_2$};
\node[right] at (p1) {$p_1'$};
\end{tikzpicture}
\quad = \quad
\begin{tikzpicture}
[baseline={([yshift=-.5ex]current bounding box.center)}]
\node[Solid](x){}
    child{node[Solid](p2){}
        child{node[Solid](n1){}}
        child{node[Solid](n2){}}
        }
    child{
        child[missing]
        child{node[Solid](n3){}}
        };
\node[above] at (x) {$x$};
\node[below] at (n1) {$n_3$};
\node[below] at (n2) {$n_2$};
\node[below] at (n3) {$n_1$};
\node[left] at (p2) {$p_1'$};
\end{tikzpicture}
\]
The transition \eqref{eq:e=sum Re} from $e_{x,p_1,r}^{1(23)}$ to $e_{x,p_2,r}^{(12)3}$ which involves a $6j$-symbol, which is equivalent to identity (ii) in Proposition \ref{prop:properties R} in terms of Clebsch-Gordan coefficients, is represented as
\[
\begin{tikzpicture}
[baseline={([yshift=-.5ex]current bounding box.center)}]
\node[Solid](x){}
    child{
        child{node[Solid](n1){}}
        child[missing]
        }
    child{node[Solid](p1){}
        child{node[Solid](n2){}}
        child{node[Solid](n3){}}
        };
\node[above] at (x) {$x$};
\node[below] at (n1) {$n_1$};
\node[below] at (n2) {$n_2$};
\node[below] at (n3) {$n_3$};
\node[right] at (p1) {$p_1'$};

\draw [->,thick] (1.25,-.8) -- ++(0:1.5)
    node[draw=none,fill=none,rectangle,above,midway]{$R_{p_1',p_2'}^{x,n_1,n_2,n_3}$};

\node[Solid](x) at (4,0) {}
    child{node[Solid](p2){}
        child{node[Solid](n1){}}
        child{node[Solid](n2){}}
        }
    child{
        child[missing]
        child{node[Solid](n3){}}
        };
\node[above] at (x) {$x$};
\node[below] at (n1) {$n_1$};
\node[below] at (n2) {$n_2$};
\node[below] at (n3) {$n_3$};
\node[left] at (p2) {$p_2'$};
\end{tikzpicture}
\]
where the coefficient $R$ is given by
\begin{equation} \label{eq:R=J}
R_{p_1',p_2'}^{x,n_1,n_2,n_3} =  R_{p_1,r_{123};p_2,r_{123}} =(-q)^{p_1'+p_2'-n_1-n_3} J_{x-n_1+n_2-n_3}(q^{2p_1'+2p_2'-2n_1-2n_3};q^2).
\end{equation}
Note that the transition from right to left involves exactly the same $6j$-symbol.
To find identities for the $6j$-symbols we can use the binary trees, and identities for these trees as explained above, without referering to the underlying eigenvectors. We obtain the following identities, which can be considered as analogs of Racah's backcoupling identity, the Biedenharn-Elliot (or pentagon) identity and the hexagon identity.
\begin{thm} \label{thm:identitiesJ}
The following identities hold:
\begin{enumerate}[(i)]
\item \hfill
$\displaystyle
R^{x,n_1,n_2,n_3}_{p_1,p_2} = \sum_{p \in \Z} R^{x,n_1,n_3,n_2}_{p_1,p}  R^{x,n_3,n_1,n_2}_{p,p_2},
$ \hfill \*\\
or in terms of $q$-Bessel functions
\[
\begin{split}
J_{r_{123}}(q^{p_1+p_2};q) =   \sum_{p \in\Z} J_{r_{132}}(q^{p+p_1};q) J_{r_{312}}(q^{p+p_2};q) q^{p},
\end{split}
\]
where $r_{ijk} = x-n_i+n_j-n_k$.
\item \hfill
$\displaystyle
R_{r_1,p_2}^{x,n_1,n_2,p_1} R_{p_1,r_2}^{x,p_2,n_3,n_4} = \sum_{p\in \Z} R_{p_1,p}^{r_1,n_2,n_3,n_4}R_{r_1,r_2}^{x,n_1,p,n_4}R_{p,p_2}^{r_2,n_1,n_2,n_3},
$ \hfill \* \\
which in terms of $q$-Bessel functions is equivalent to the product formula
\[
\begin{split}
J_{\nu+\mu_1}(q^{P-Q};q) J_{\nu+\mu_2}(q^{Q-R};q)
= \sum_{\mu \in \Z} A_{P,Q,R}^{\mu_1,\mu_2,\mu} \, J_{\nu+\mu}(q^{P-R};q)
\end{split}
\]
where $P,Q,R,\nu,\mu_1,\mu_2 \in \Z$ and
\[
A_{P,Q,R}^{\mu_1,\mu_2,\mu} = (-1)^{\mu_1+\mu_2}q^{\mu-\frac12(\mu_1+\mu_2)}J_{\mu_2-\mu_1+P-Q}(q^{\mu-\mu_1};q) J_{\mu_1-\mu_2+Q-R}(q^{\mu-\mu_2};q).
\]

\item \hfill
$ \displaystyle
\sum_{r \in \Z} R_{p_2,r}^{x,p_1,n_3,n_4} R_{p_3,p_1}^{r,n_2,n_1,n_3} R_{p_4,r}^{x,p_3,n_2,n_4} =
\sum_{r \in \Z} R_{r,p_1}^{x,n_1,n_2,p_2}R_{p_2,p_4}^{r,n_2,n_4,n_3}R_{r,p_3}^{x,n_1,n_3,p_4},
$ \hfill \* \\
or in terms of $q$-Bessel functions,
\[
\begin{split}
&\sum_{r\in\Z} (-1)^{p_2+p_4}q^{r -n_4+\frac12(p_2+p_4)} J_{r-n_2+n_1-n_3}(q^{p_1+p_3-n_2-n_3};q) \\
& \qquad \times J_{x-p_1+n_3-n_4}(q^{r+p_2-p_1-n_4};q) J_{x-p_3+n_2-n_4}(q^{r+p_4-p_3-n_4};q) \\
 = &\, \mathrm{idem}\big( (n_1,n_2,p_1,p_3) \leftrightarrow (n_4,n_3,p_2,p_4) \big) .
\end{split}
\]
\end{enumerate}
\end{thm}
Here `idem' means that the same expression is inserted but with the parameters interchanged as indicated.
\begin{proof}
The first identity follows from
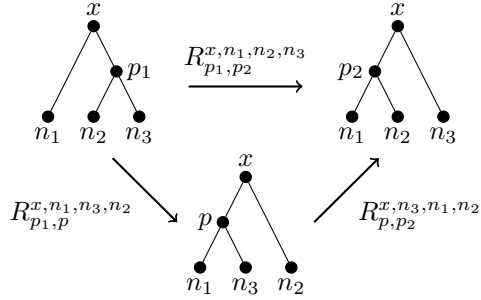
\begin{center}
\begin{tikzpicture}

\node[Solid](x){}
    child{
        child{node[Solid](n1){}}
        child[missing]
        }
    child{node[Solid](p1){}
        child{node[Solid](n2){}}
        child{node[Solid](n3){}}
        };
\node[above] at (x) {$x$};
\node[below] at (n1) {$n_1$};
\node[below] at (n2) {$n_2$};
\node[below] at (n3) {$n_3$};
\node[right] at (p1) {$p_1$};


\node[Solid](x) at (4,0) {}
    child{node[Solid](p2){}
        child{node[Solid](n1){}}
        child{node[Solid](n2){}}
        }
    child{
        child[missing]
        child{node[Solid](n3){}}
        };
\node[above] at (x) {$x$};
\node[below] at (n1) {$n_1$};
\node[below] at (n2) {$n_2$};
\node[below] at (n3) {$n_3$};
\node[left] at (p2) {$p_2$};


\node[Solid](x) at (2,-2) {}
    child{node[Solid](p){}
        child{node[Solid](n1){}}
        child{node[Solid](n2){}}
        }
    child{
        child[missing]
        child{node[Solid](n3){}}
        };
\node[above] at (x) {$x$};
\node[below] at (n1) {$n_1$};
\node[below] at (n2) {$n_3$};
\node[below] at (n3) {$n_2$};
\node[left] at (p) {$p$};

\draw [->,thick] (1.25,-.8) -- ++(0:1.5)
    node[draw=none,fill=none,rectangle,above,midway]{$R_{p_1,p_2}^{x,n_1,n_2,n_3}$};
\draw [->,thick] (.25,-1.75) -- ++((-45:1.2)
    node[draw=none,fill=none,rectangle,below left,midway]{$R_{p_1,p}^{x,n_1,n_3,n_2}$};
\draw [<-,thick] (3.75,-1.75) -- ++((-135:1.2)
    node[draw=none,fill=none,rectangle,below right,midway]{$R_{p,p_2}^{x,n_3,n_1,n_2}$};
\end{tikzpicture}
\end{center}
The second identity is
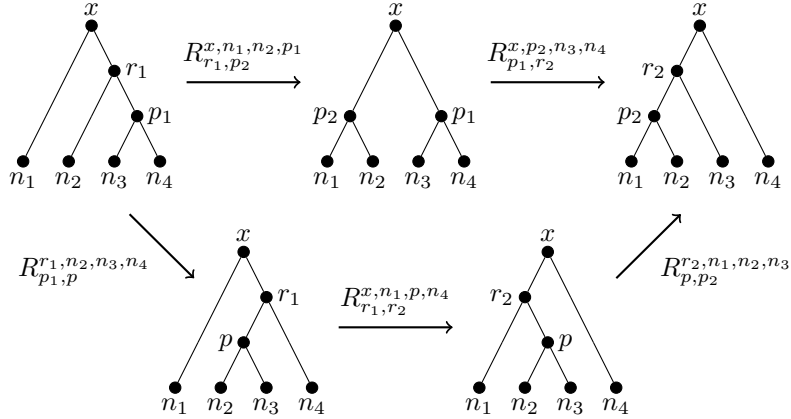
\begin{center}
\begin{tikzpicture}
\node[Solid](x) at (0,0) {}
    child{
        child{
            child{node[Solid](n1){}}
            child[missing]
            }
        child[missing]
        }
    child{node[Solid](r1){}
        child{
            child{node[Solid](n2){}}
            child[missing]
            }
        child{node[Solid](p1){}
            child{node[Solid](n3){}}
            child{node[Solid](n4){}}
            }
        };
\node[above] at (x) {$x$};
\node[below] at (n1) {$n_1$};
\node[below] at (n2) {$n_2$};
\node[below] at (n3) {$n_3$};
\node[below] at (n4) {$n_4$};
\node[right] at (r1) {$r_1$};
\node[right] at (p1) {$p_1$};


\node[Solid](x) at (4,0) {}
    child{
        child{node[Solid](p2){}
            child{node[Solid](n1){}}
            child{node[Solid](n2){}}
            }
        child[missing]
        }
    child{
        child[missing]
        child{node[Solid](p1){}
            child{node[Solid](n3){}}
            child{node[Solid](n4){}}
            }
        };

\node[above] at (x) {$x$};
\node[below] at (n1) {$n_1$};
\node[below] at (n2) {$n_2$};
\node[below] at (n3) {$n_3$};
\node[below] at (n4) {$n_4$};
\node[left] at (p2) {$p_2$};
\node[right] at (p1) {$p_1$};

\node[Solid](x) at (8,0) {}
    child{node[Solid](r2){}
        child{node[Solid](p2){}
            child{node[Solid](n1){}}
            child{node[Solid](n2){}}
            }
        child{
            child[missing]
            child{node[Solid](n3){}}
            }
        }
    child{
        child[missing]
        child{
            child[missing]
            child{node[Solid](n4){}}
            }
        };

\node[above] at (x) {$x$};
\node[below] at (n1) {$n_1$};
\node[below] at (n2) {$n_2$};
\node[below] at (n3) {$n_3$};
\node[below] at (n4) {$n_4$};
\node[left] at (p2) {$p_2$};
\node[left] at (r2) {$r_2$};


\node[Solid](x) at (2,-3) {}
    child{
        child{
            child{node[Solid](n1){}}
            child[missing]
            }
        child[missing]
        }
    child{node[Solid](r1){}
        child{node[Solid](p){}
            child{node[Solid](n2){}}
            child{node[Solid](n3){}}
            }
        child{
            child[missing]
            child{node[Solid](n4){}}
            }
        };

\node[above] at (x) {$x$};
\node[below] at (n1) {$n_1$};
\node[below] at (n2) {$n_2$};
\node[below] at (n3) {$n_3$};
\node[below] at (n4) {$n_4$};
\node[right] at (r1) {$r_1$};
\node[left] at (p) {$p$};


\node[Solid](x) at (6,-3) {}
    child{node[Solid](r2){}
        child{
            child{node[Solid](n1){}}
            child[missing]
            }
        child{node[Solid](p){}
            child{node[Solid](n2){}}
            child{node[Solid](n3){}}
            }
        }
    child{
        child[missing]
        child{
            child[missing]
            child{node[Solid](n4){}}
            }
        };

\node[above] at (x) {$x$};
\node[below] at (n1) {$n_1$};
\node[below] at (n2) {$n_2$};
\node[below] at (n3) {$n_3$};
\node[below] at (n4) {$n_4$};
\node[left] at (r2) {$r_2$};
\node[right] at (p) {$p$};


\draw [->,thick] (1.25,-.75) -- ++(0:1.5)
    node[draw=none,fill=none,rectangle,above,midway]{$R_{r_1,p_2}^{x,n_1,n_2,p_1}$};
\draw [->,thick] (5.25,-.75) -- ++(0:1.5)
    node[draw=none,fill=none,rectangle,above,midway]{$R_{p_1,r_2}^{x,p_2,n_3,n_4}$};

\draw [->,thick] (.5,-2.5) -- ++(-45:1.2)
    node[draw=none,fill=none,rectangle,below left,midway]{$R_{p_1,p}^{r_1,n_2,n_3,n_4}$};
\draw [->,thick] (3.25,-4) -- ++(0:1.5)
    node[draw=none,fill=none,rectangle,above,midway]{$R_{r_1,r_2}^{x,n_1,p,n_4}$};
\draw [<-,thick] (7.75,-2.5) -- ++((-135:1.2)
    node[draw=none,fill=none,rectangle,below right,midway]{$R_{p,p_2}^{r_2,n_1,n_2,n_3}$};
\end{tikzpicture}
\end{center}
The corresponding identity for $q$-Bessel functions is obtained by substituting
\begin{gather*}
r_1-n_1=P, \quad p_1-p_2=Q, \quad n_4-r_2=R, \quad \nu=x-n_1, \\ \mu_1=n_2-p_1, \quad \mu_2 = n_1-p_2+n_3-n_4, \quad \mu=p-n_4.
\end{gather*}

The third identity is
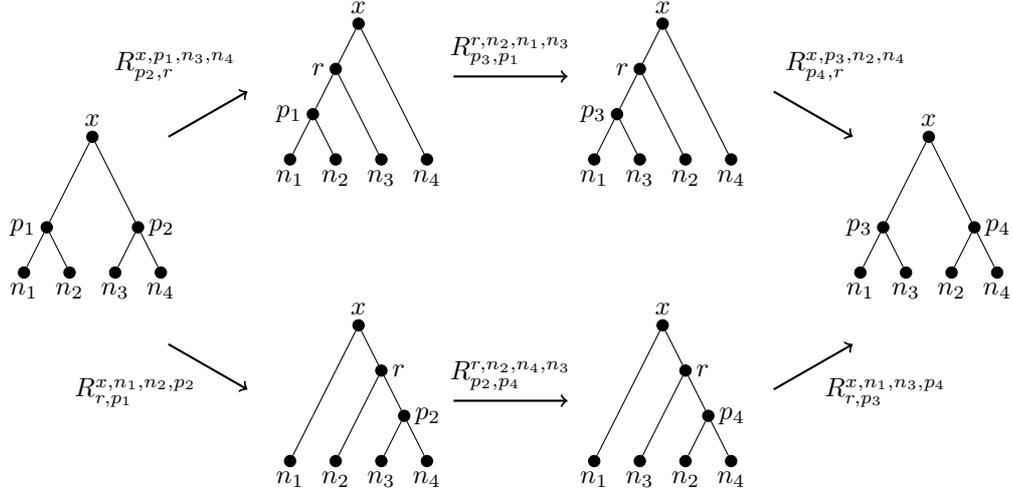
\begin{center}
\begin{tikzpicture}

\node[Solid](x) at (-.5,0) {}
    child{
        child{node[Solid](p1){}
            child{node[Solid](n1){}}
            child{node[Solid](n2){}}
            }
        child[missing]
        }
    child{
        child[missing]
        child{node[Solid](p2){}
            child{node[Solid](n3){}}
            child{node[Solid](n4){}}
            }
        };

\node[above] at (x) {$x$};
\node[below] at (n1) {$n_1$};
\node[below] at (n2) {$n_2$};
\node[below] at (n3) {$n_3$};
\node[below] at (n4) {$n_4$};
\node[left] at (p1) {$p_1$};
\node[right] at (p2) {$p_2$};


\node[Solid](x) at (3,1.5) {}
    child{node[Solid](r){}
        child{node[Solid](p1){}
            child{node[Solid](n1){}}
            child{node[Solid](n2){}}
            }
        child{
            child[missing]
            child{node[Solid](n3){}}
            }
        }
    child{
        child[missing]
        child{
            child[missing]
            child{node[Solid](n4){}}
            }
        };

\node[above] at (x) {$x$};
\node[below] at (n1) {$n_1$};
\node[below] at (n2) {$n_2$};
\node[below] at (n3) {$n_3$};
\node[below] at (n4) {$n_4$};
\node[left] at (p1) {$p_1$};
\node[left] at (r) {$r$};


\node[Solid](x) at (7,1.5) {}
    child{node[Solid](r){}
        child{node[Solid](p3){}
            child{node[Solid](n1){}}
            child{node[Solid](n3){}}
            }
        child{
            child[missing]
            child{node[Solid](n2){}}
            }
        }
    child{
        child[missing]
        child{
            child[missing]
            child{node[Solid](n4){}}
            }
        };

\node[above] at (x) {$x$};
\node[below] at (n1) {$n_1$};
\node[below] at (n2) {$n_2$};
\node[below] at (n3) {$n_3$};
\node[below] at (n4) {$n_4$};
\node[left] at (p3) {$p_3$};
\node[left] at (r) {$r$};


\node[Solid](x) at (10.5,0) {}
    child{
        child{node[Solid](p3){}
            child{node[Solid](n1){}}
            child{node[Solid](n3){}}
            }
        child[missing]
        }
    child{
        child[missing]
        child{node[Solid](p4){}
            child{node[Solid](n2){}}
            child{node[Solid](n4){}}
            }
        };
\node[above] at (x) {$x$};
\node[below] at (n1) {$n_1$};
\node[below] at (n2) {$n_2$};
\node[below] at (n3) {$n_3$};
\node[below] at (n4) {$n_4$};
\node[left] at (p3) {$p_3$};
\node[right] at (p4) {$p_4$};


\node[Solid](x) at (3,-2.5) {}
    child{
        child{
            child{node[Solid](n1){}}
            child[missing]
            }
        child[missing]
        }
    child{node[Solid](r){}
        child{
            child{node[Solid](n2){}}
            child[missing]
            }
        child{node[Solid](p2){}
            child{node[Solid](n3){}}
            child{node[Solid](n4){}}
            }
        };
\node[above] at (x) {$x$};
\node[below] at (n1) {$n_1$};
\node[below] at (n2) {$n_2$};
\node[below] at (n3) {$n_3$};
\node[below] at (n4) {$n_4$};
\node[right] at (p2) {$p_2$};
\node[right] at (r) {$r$};


\node[Solid](x) at (7,-2.5) {}
    child{
        child{
            child{node[Solid](n1){}}
            child[missing]
            }
        child[missing]
        }
    child{node[Solid](r){}
        child{
            child{node[Solid](n3){}}
            child[missing]
            }
        child{node[Solid](p4){}
            child{node[Solid](n2){}}
            child{node[Solid](n4){}}
            }
        };
\node[above] at (x) {$x$};
\node[below] at (n1) {$n_1$};
\node[below] at (n2) {$n_2$};
\node[below] at (n3) {$n_3$};
\node[below] at (n4) {$n_4$};
\node[right] at (p4) {$p_4$};
\node[right] at (r) {$r$};


\draw [->,thick] (.5,0) -- ++((30:1.2)
    node[draw=none,fill=none,rectangle,above left]{$R_{p_2,r}^{x,p_1,n_3,n_4}$};
\draw [->,thick] (4.25,.8) -- ++((0:1.5)
    node[draw=none,fill=none,rectangle,above,midway]{$R_{p_3,p_1}^{r,n_2,n_1,n_3}$};
\draw [<-,thick] (9.5,0) -- ++((150:1.2)
    node[draw=none,fill=none,rectangle,above right]{$R_{p_4,r}^{x,p_3,n_2,n_4}$};
\draw [->,thick] (.5,-2.75) -- ++((-30:1.2)
    node[draw=none,fill=none,rectangle,below left,midway]{$R_{r,p_1}^{x,n_1,n_2,p_2}$};
\draw [->,thick] (4.25,-3.5) -- ++((0:1.5)
    node[draw=none,fill=none,rectangle,above,midway]{$R_{p_2,p_4}^{r,n_2,n_4,n_3}$};
\draw [<-,thick] (9.5,-2.75) -- ++((-150:1.2)
    node[draw=none,fill=none,rectangle,below right,midway]{$R_{r,p_3}^{x,n_1,n_3,p_4}$};

\end{tikzpicture}
\end{center}

\end{proof}

\begin{remark}\*
\begin{enumerate}[(i)]
\item The $q$-Hankel transform of a function $f \in L^2(q^\Z;q^x)$ is defined by
\[
(H_\nu f)(n) = \sum_{x \in \Z} f(q^x) J_\nu(q^{x+n};q) q^x , \qquad n \in \Z.
\]
Identity (i) of Theorem \ref{thm:identitiesJ} shows that the $q$-Hankel transform maps an orthogonal basis of $q$-Bessel functions to another orthogonal basis of $q$-Bessel functions, which implies a factorization of the $q$-Hankel transform: $H_{r_{123}} = H_{r_{312}} H_{r_{132}}$.
\item Identity (ii), the product formula for $q$-Bessel functions, has appeared before in the literature; representation theoretic proofs are given by Koelink in \cite[Cor.6.5]{Koe94}, and Kalnins, Miller and Mukherjee in \cite[(3.20)]{KalMilMu}. A direct analytic proof is given by Koelink and Swarttouw in \cite{KoeSw}.
\item It is well known that the hexagon identity for classical $6j$-symbols can be interpreted as a quantum Yang-Baxter equation. Here we obtain an infinite dimensional solution: for $u,v \in \Z$, define a unitary operator $\mathcal R(u,v): \ell^2(\Z)\tensor \ell^2(\Z) \to \ell^2(\Z)\tensor \ell^2(\Z)$ by
\[
\mathcal R(u,v) (e_{x-a} \tensor e_{b-x} ) = \sum_{y \in \Z} R_{x,y}^{u,a,v,b} \, e_{b-y} \tensor e_{y-a},\qquad a,b,x \in \Z,
\]
where $\{ e_x \mid x \in \Z\}$ is the standard orthonormal basis for $\ell^2(\Z)$. Then the hexagon identity says that the operator $\mathcal R$ satisfies
\[
\mathcal R_{12}(u,w) \mathcal R_{13}(v,w) \mathcal R_{23}(u,v) = \mathcal R_{23}(u,v) \mathcal R_{13}(v,w) \mathcal R_{12}(u,w)
\]
as an operator identity on $\ell^2(\Z) \tensor \ell^2(\Z) \tensor \ell^2(\Z)$.
\end{enumerate}
\end{remark}

\section{$3nj$-symbols and multivariate $q$-Bessel functions} \label{sec:3nj-symbols}
We consider certain $3nj$-symbols and show that these can be considered as multivariate $q$-Bessel functions, which are limits of the multivariate Askey-Wilson polynomials introduced by Gasper and Rahman in \cite{GR05}.
In this section we use the following notation. For $v = (v_1,v_2,\ldots,v_{d-1},v_{d})$ we define $|v|=\sum_{j=1}^d v_j$ and $\hat v = (v_d,v_{d-1},\ldots,v_2,v_1)$. For some function $f:\Z^d \to \C$ we set
\[
\sum_{x} f(x) = \sum_{x_d \in \Z} \cdots \sum_{x_1 \in \Z} f(x_1,\ldots,x_d),
\]
provided the sum converges.

\subsection{Multivariate $q$-Bessel functions}
Let $d \in \N_{\geq 1}$. For $\nu = (\nu_0,\ldots,\nu_{d+1}) \in \Z^{d+2}$, we define $q$-Bessel functions in the variables $x=(x_1,\ldots,x_d), \la=(\la_1,\ldots,\la_d) \in \Z^d$ by
\begin{equation} \label{eq:def multivariable Jnu}
J_\nu(x,\la) = \prod_{j=1}^d J_{\nu_j- x_{j+1} - \la_{j-1} }(q^{x_j-x_{j+1} + \la_j- \la_{j-1}};q),
\end{equation}
where $\la_0=\nu_0$ and $x_{d+1}=\nu_{d+1}$. Occasionally we will use the notation $J_\nu(x,\la;q)$ to stress the dependence on $q$.

\begin{thm} \label{thm:properties multivariable J}
The multivariate $q$-Bessel functions have the following properties:
\begin{enumerate}[(i)]
\item Orthogonality relations:
\[
\sum_{x} J_\nu(x,\la) J_\nu(x,\la') q^{x_1} = \de_{\la ,\la'} q^{\nu_{d+1}+\nu_0 - \la_d}, \qquad \la,\la' \in \Z^d.
\]
\item Self-duality: $J_\nu(x,\la) = J_{\hat \nu}(\hat \la,\hat x)$.
\end{enumerate}
\end{thm}
\begin{proof}
The self-duality property follows directly from \eqref{eq:def multivariable Jnu}.
The orthogonality relations follow by induction using the $q$-Hankel orthogonality relations from Theorem \ref{thm:orthogonality}, which can be written as
\begin{multline} \label{eq:orthogonality2}
\sum_{x_j \in \Z} J_{\nu_j - x_{j+1} - \la_{j-1}} (q^{x_j - x_{j+1} + \la_j - \la_{j-1}};q) J_{\nu_j - x_{j+1} - \la_{j-1}} (q^{x_j - x_{j+1} + \la_j' - \la_{j-1}};q) q^{x_j}\\
 = \de_{\la_j,\la_j'} q^{x_{j+1}-\la_j+\la_{j-1}}.
\end{multline}
Define for $k=1,\ldots,d+1$,
\[
J_\nu^{(k)}(x,\la) = \prod_{j=k}^d J_{\nu_j- x_{j+1} - \la_{j-1} }(q^{\frac12(x_j-x_{j+1} + \la_j- \la_{j-1})};q),
\]
the empty product being equal to $1$. Note that $J_\nu^{(1)} = J_\nu$ and
\begin{equation} \label{eq:k+1 -> k}
J_{\nu_{k}-x_{k+1}-\la_{k-1}}(q^{x_{k} - x_{k+1} + \la_{k}-\la_{k-1}};q) J_\nu^{(k+1)}(x,\la) = J_\nu^{(k)}(x,\la).
\end{equation}
We will show that
\begin{equation} \label{eq:IH}
\sum_{x_k \in \Z} \cdots \sum_{x_1 \in \Z} J_\nu(x,\la) J_\nu(x,\la') q^{x_1} = \de_{\la_1,\la_1'}  \cdots \de_{\la_k,\la_k'} q^{x_{k+1} - \la_k + \la_0} \, J_\nu^{(k+1)}(x,\la) J_\nu^{(k+1)}(x,\la').
\end{equation}
For $k=1$ \eqref{eq:IH} follows directly from \eqref{eq:orthogonality2}.
Now assume that \eqref{eq:IH} holds for a certain $k$, then by \eqref{eq:orthogonality2} and \eqref{eq:k+1 -> k},
\[
\begin{split}
\sum_{x_{k+1} \in \Z} \cdots \sum_{x_1 \in \Z} & J_\nu(x,\la) J_\nu(x,\la') q^{x_1} \\
& = \de_{\la_1,\la_1'}  \cdots \de_{\la_k,\la_k'} \sum_{x_{k+1} \in \Z} J_\nu^{(k+1)}(x,\la) J_\nu^{(k+1)}(x,\la')q^{x_{k+1} - \la_k + \la_0} \\
& = \de_{\la_1,\la_1'}  \cdots \de_{\la_{k+1},\la_{k+1}'}J_\nu^{(k+2)}(x,\la) J_\nu^{(k+2)}(x,\la')q^{x_{k+2} - \la_{k+1} + \la_0},
\end{split}
\]
which proves the orthogonality relations.
\end{proof}

Next we show that the multivariate $q$-Bessel functions can be considered as limit cases of multivariate Askey-Wilson polynomials.
The 1-variable Askey-Wilson polynomials are defined by
\[
p_n(x;a,b,c,d \,|\, q) = \frac{(ab,ac,ad;q)_n }{a^n} \rphis{4}{3}{q^{-n}, abcdq^{n-1}, ax, a/x }{ab, ac, ad}{q,q},
\]
which are polynomials in $x+x^{-1}$ of degree $n$, and they are symmetric in the parameters $a,b,c,d$. Using notation as in \cite{Il} the multivariable Askey-Wilson polynomials are defined as follows. Let $n = (n_1,\ldots,n_d) \in \N^d$ and $x=(x_1,\ldots,x_d) \in (\C^\times)^d$, then the $d$-variable Askey-Wilson polynomials are defined by
\begin{equation} \label{eq:d-var AWpol}
P_d(n;x;\al\,|\,  q) = \prod_{j=1}^d p_{n_j}\left(x_j;\al_{j} q^{N_{j-1}}, \frac{ \al_j}{\al_0^2} q^{N_{j-1}}, \frac{\al_{j+1}}{\al_j} x_{j+1}, \frac{\al_{j+1}}{\al_j} x_{j+1}^{-1} \,|\, q\right),
\end{equation}
where $N_j=\sum_{k=1}^j n_k$, $N_0=0$, $\al=(\al_0,\ldots, \al_{d+2}) \in \C^{d+3}$, $x_{d+1}=\al_{d+2}$. These are polynomials in the variables $x_1+x_1^{-1}, \ldots, x_d+x_d^{-1}$ of degree $|n|=N_d$.

\begin{prop}
Let $\la=(\la_1,\ldots,\la_d) \in \Z^d$, $\nu=(\nu_0,\ldots,\nu_{d+1}) \in \Z^{d+2}$ and define
\[
\begin{aligned}
\al(m) &= \left( q^{-m}, q^{\frac12\nu_0}, q^{\frac12\nu_1-m},\ldots,  q^{\frac12\nu_j-jm},\ldots, q^{\frac12\nu_{d}-dm}, q^{\nu_{d+1}+m} \right) \in \C^{d+3},\\
x(m) &= \left( q^{-\frac12\nu_0-x_1+m}, q^{\frac12\nu_1-\nu_0-x_2+m}, \ldots, q^{\frac12\nu_{d-1}-\nu_0-x_d+m} \right)\in \C^d, \\
 \la+m &= (\la_1+m, \ldots, \la_d+m) \in \N^{d},\\
C_m(x;\la;\al) &= \prod_{j=1}^d  q^{(\frac12\nu_{j-1}-\nu_j+\nu_0+x_{j+1}-m)(\la_j+m)}  \left( q^{\nu_j-\nu_{j-1}-2m} ;q\right)_{\la_j+m},
\end{aligned}
\]
then
\[
\lim_{m \to \infty} \frac{  P_d(\la+m;x(m);\al(m)\,|\, q)}{C_m(x;\la;\al)} = (q;q)_\infty^d \left( \prod_{j=1}^d q^{\frac12(x_{j+1}-x_{j}+\La_{j+1}-\La_{j})(\nu_j-x_{j+1}-\La_{j-1})} \right) J_\nu(x,\La),
\]
where $\La=(\La_1,\ldots,\La_d)$ with $\La_j=\nu_0-\sum_{k=1}^{j} \la_k$ and $\La_0=\nu_0$.
\end{prop}
\begin{proof}
First we substitute
\begin{align*}
\al_0 &\mapsto q^{-m}, & n_j &\mapsto \la_j + m, && j=1,\ldots,d \\
x_j &\mapsto x_j q^m, & \al_j &\mapsto \al_j q^{m(j-1)}, && j=1,\ldots,d+1,
\end{align*}
in \eqref{eq:d-var AWpol} (recall, $x_{d+1}=\al_{d+2}$). The $_4\varphi_3$-part of the $j$-th factor $p_{n_j}$ is
\[
\begin{split}
\rphis{4}{3}{q^{-\la_j-m}, \al_{j+1}^2 q^{2(\sum_{k=1}^{j-1}\la_{k})+\la_j-1+m}, \frac{\al_{j+1}}{\al_j} x_{j+1}x_j q^{m},  \frac{\al_{j+1}x_{j+1}}{\al_j x_j}q^{-m} } { \frac{\al_{j+1}^2}{\al_j^2}q^{-2m}, \al_{j+1}x_{j+1}q^{2m+\sum_{k=1}^{j-1}\la_k}, \al_{j+1}x_{j+1}q^{\sum_{k=1}^{j-1}\la_k} }{q,q},
\end{split}
\]
where the empty sum equals $0$. Letting $m \to \infty$ this function tends to
\[
\rphis{1}{1}{0}{\al_{j+1}x_{j+1}q^{\sum_{k=1}^{j-1}\la_k}}{q, \frac{\al_j x_{j+1}}{\al_{j+1}x_j}q^{1-\la_j} }.
\]
Finally we substitute
\[
\al_j \mapsto q^{\frac 12 \nu_{j-1}}, \quad x_j \mapsto q^{\frac12\nu_{j-1}-\nu_0 - x_j},
\]
and set $\nu_0-\sum_{k=1}^{j} \la_k = \La_j$ for $j=0,\ldots,d$, then we have
\[
\rphis{1}{1}{0}{q^{\nu_j - x_{j+1} - \La_{j-1}}}{q, q^{x_{j+1}-x_{j}+\La_{j} - \La_{j-1}}},
\]
which we recognize as the $_1\varphi_1$-part of the $j$-th factor of the multivariate $q$-Bessel function $J_\nu(x,\La)$, see \eqref{eq:def multivariable Jnu}.
\end{proof}

\subsection{$3nj$-symbols}
Let $k \in \N_{\geq 1}$, and let $\br,\bs \in \Z^{k}$, $\bn \in \Z^{k+2}$. We define the $3nj$-symbols $R_{\br,\bs}^{x;\bn}$ to be the coupling coefficients between two specific binary trees corresponding to $(k+2)$-fold tensor product representations. We will use the following notation:
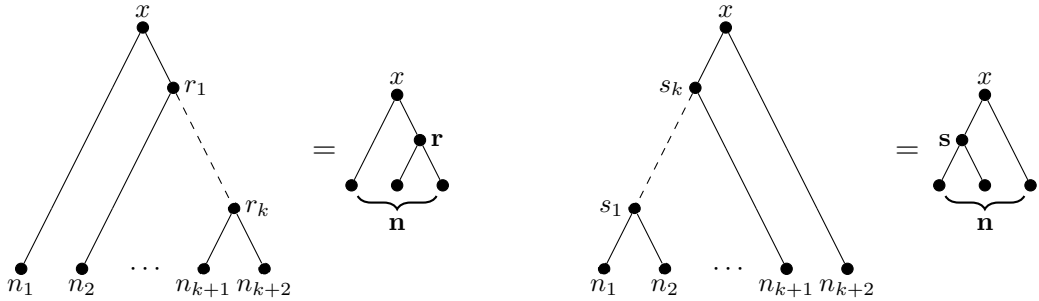
\[
\begin{tikzpicture}
[baseline={([yshift=-.5ex]current bounding box.center)},
level distance=.8cm,
sibling distance=.8cm]
\node[Solid](x) {}
    child{
        child{
            child{
                child{node[Solid](n1) {}}
                child[missing]
                }
            child[missing]
            }
        child[missing]
        }
    child{node[Solid](r1){}
        child{
            child{
                child{node[Solid](n2){}}
                child[missing]
                }
            child[missing]
            }
        child{
            edge from parent [dashed]
            child[missing]
            child{node[Solid](rk){}
                child{node[Solid](nk+1){}
                    edge from parent [solid]}
                child{node[Solid](nk+2){}
                    edge from parent [solid]}
                }
            }
        };

\node[above] at (x) {$x$};
\node[below] at (n1) {$n_1$};
\node[below] at (n2) {$n_2$};
\node[below] at (nk+1) {$n_{k+1}$};
\node[below] at (nk+2) {$n_{k+2}$};
\node[right] at (r1) {$r_1$};
\node[right] at (rk) {$r_k$};

\node [right=.4cm of n2] {$\cdots$};

\end{tikzpicture}
=
\begin{tikzpicture}
[baseline={([yshift=-.5ex]current bounding box.center)}]
\node(x)[Solid]{}
    child{
        child{node[Solid](n1){}}
        child[missing]
        }
    child{node[Solid](r){}
        child{node[Solid]{}}
        child{node[Solid](n3){}}
        };
\node [above] at (x) {$x$};
\node [right] at (r) {$\br$};
\draw [thick,decorate,decoration={brace,mirror,amplitude=5pt,raise=4pt}] (n1) -- (n3)
        node[midway,below=8pt]{$\bn$};
\end{tikzpicture}
\qquad \qquad
\begin{tikzpicture}
[baseline={([yshift=-.5ex]current bounding box.center)},
level distance=.8cm,
sibling distance=.8cm]
\node[Solid](x) {}
    child{node[Solid](sk){}
        child{
            edge from parent [dashed]
            child{node[Solid](s1){}
                child{node[Solid](n1){}
                    edge from parent [solid]}
                child{node[Solid](n2){}
                    edge from parent [solid]}
                }
            child[missing]
            }
        child{
            child[missing]
            child{
                child[missing]
                child{node[Solid](nk+1){}}
                }
            }
        }
    child{
        child[missing]
        child{
            child[missing]
            child{
                child[missing]
                child{node[Solid](nk+2){}}
                }
            }
        };
\node[above] at (x) {$x$};
\node[below] at (n1) {$n_1$};
\node[below] at (n2) {$n_2$};
\node[below] at (nk+1) {$n_{k+1}$};
\node[below] at (nk+2) {$n_{k+2}$};
\node[left] at (s1) {$s_1$};
\node[left] at (sk) {$s_k$};

\node [right=.4cm of n2] {$\cdots$};
\end{tikzpicture}
=
\begin{tikzpicture}
[baseline={([yshift=-.5ex]current bounding box.center)}]
\node[Solid](x){}
    child{node[Solid](s){}
        child{node[Solid](n1){}}
        child{node[Solid]{}}
        }
    child{
        child[missing]
        child{node[Solid](n3){}}
        };
\node[above] at (x) {$x$};
\node[left] at (s) {$\bs$};
\draw [thick,decorate,decoration={brace,mirror,amplitude=5pt,raise=4pt}] (n1) -- (n3)
        node[midway,below=8pt]{$\bn$};
\end{tikzpicture}
\]
Note that a node with a bold symbol represents several nodes, and that that the label $\br$ (respectively $\bs$) on the right (left) of a node means that all branches `hang' on the right (left) edge. The $3nj$-symbols $R_{\br,\bs}^{x,\bn}$ are defined by
\[
\begin{tikzpicture}
[baseline={([yshift=-.5ex]current bounding box.center)}]
\node[Solid](x){}
    child{
        child{node[Solid](n1){}}
        child[missing]
        }
    child{node[Solid](r){}
        child{node[Solid]{}}
        child{node[Solid](n3){}}
        };
\node[above] at (x) {$x$};
\node[right] at (r) {$\br$};
\draw [thick,decorate,decoration={brace,mirror,amplitude=5pt,raise=4pt}] (n1) -- (n3)
        node[midway,below=8pt]{$\bn$};
\end{tikzpicture}
= \sum_{\bs} R_{\br,\bs}^{x,\bn} \
\begin{tikzpicture}
[baseline={([yshift=-.5ex]current bounding box.center)}]
\node[Solid](x){}
    child{node[Solid](s){}
        child{node[Solid](n1){}}
        child{node[Solid]{}}
        }
    child{
        child[missing]
        child{node[Solid](n3){}}
        };
\node[above] at (x) {$x$};
\node[left] at (s) {$\bs$};
\draw [thick,decorate,decoration={brace,mirror,amplitude=5pt,raise=4pt}] (n1) -- (n3)
        node[midway,below=8pt]{$\bn$};
\end{tikzpicture}
\]
and we will denote the corresponding transition again by an arrow. Note that for $k=1$ we have $R_{r,s}^{x,\bn}=R_{r,s}^{x,n_1,n_2,n_3}$.
\begin{prop} \label{prop:properties R multivariate}
The coefficients $R_{\br,\bs}^{x,\bn}$ have the following properties:
\begin{enumerate}[(i)]
\item Orthogonality relations: $\displaystyle \sum_{\br} R_{\br,\bs}^{x,\bn} R_{\br,\bs'}^{x,\bn} = \de_{\bs,\bs'}$
\item Duality: $R_{\br,\bs}^{x,\bn} = R_{\hat \bs,\hat \br}^{x,\hat\bn}$.
\end{enumerate}
\end{prop}
\begin{proof}
The coefficients $R$ are the matrix coefficients of a unitary operator, which implies the orthogonality relations. The duality property is a consequence of the identity
\[
\begin{tikzpicture}
[baseline={([yshift=-.5ex]current bounding box.center)}]
\node[Solid](x){}
    child{
        child{node[Solid](n1){}}
        child[missing]
        }
    child{node[Solid](r){}
        child{node[Solid]{}}
        child{node[Solid](n3){}}
        };
\node[above] at (x) {$x$};
\node[right] at (r) {$\br$};
\draw [thick,decorate,decoration={brace,mirror,amplitude=5pt,raise=4pt}] (n1) -- (n3)
        node[midway,below=8pt]{$\bn$};
\end{tikzpicture} =
\begin{tikzpicture}
[baseline={([yshift=-.5ex]current bounding box.center)}]
\node[Solid](x){}
    child{node[Solid](r){}
        child{node[Solid](n1){}}
        child{node[Solid]{}}
        }
    child{
        child[missing]
        child{node[Solid](n3){}}
        };
\node[above] at (x) {$x$};
\node[left] at (r) {$\hat{\br}$};
\draw [thick,decorate,decoration={brace,mirror,amplitude=5pt,raise=4pt}] (n1) -- (n3)
        node[midway,below=8pt]{$\hat\bn$};
\end{tikzpicture}
\]
which follows from repeated application of \eqref{eq:tree flip}.
\end{proof}

\begin{thm}
For $i=1,2$ let $k_i \in \N_{\geq 1}$, $\bn_i \in \Z^{k_i+1}$ and $\br_i,\bs_i \in \Z^{k_i}$. Let $k=k_1+k_2$, $\bn=(\bn_1,\bn_2)$, $\br=(\br_1,\br_2)$, $\bs=(\bs_1,\bs_2)$, then
\[
R_{\br,\bs}^{x,\bn} = R_{\br_1,\bs_1}^{x,(\bn_1,r_{k_1+1})} R_{\br_2,\bs_2}^{x,(s_{k_1},\bn_2)}.
\]
As a consequence,
\[
R_{\br,\bs}^{x,\bn} = \prod_{j=1}^{k} R_{r_j,s_j}^{x,s_{j-1},n_{j+1},r_{j+1}},
\]
where $s_0=n_1$ and $r_{k+1}=n_{k+2}$.
\end{thm}
\begin{proof}
The first identity follows from
\[
\begin{tikzpicture}
[baseline={([yshift=-.5ex]current bounding box.center)}]

\begin{scope}[xshift=-110pt]
\node[Solid](x){}
    child{
        child{
            child{node[Solid](n1){}}
            child[missing]
            }
        child[missing]
        }
    child{node[Solid](r1){}
        child{
            child{node[Solid](n2){}}
            child[missing]
        }
        child{node[Solid](r2){}
            child[thin]{node[Solid](n3){}}
            child[thin]{node[Solid](n4){}}
            }
        };
\node[above] at (x) {$x$};
\node[right] at (r2) {$\br_2$};
\node[right] at (r1) {$\br_1$};
\draw [thick,decorate,decoration={brace,mirror,amplitude=3pt,raise=4pt}] (n1) -- (n2)
        node[midway,below=8pt]{$\bn_1$};
\draw [thick,decorate,decoration={brace,mirror,amplitude=3pt,raise=4pt}] (n3) -- (n4)
        node[midway,below=8pt]{$\bn_2$};

\draw [->,thick] ($(n4)+(.25,1)$) -- ++(0:1.5)
            node[midway,above]{$R_{\br_1,\bs_1}^{x,(\bn_1,r_{k_1+1})}$};
\end{scope}

\begin{scope}
\node[Solid](x){}
    child{
        child{node[Solid](s1){}
            child{node[Solid](n1){}}
            child{node[Solid](n2){}}
            }
        child[missing]
        }
    child{
        child[missing]
        child{node[Solid](r2){}
            child{node[Solid](n3){}}
            child{node[Solid](n4){}}
            }
        };
\node[above] at (x) {$x$};
\node[right] at (r2) {$\br_2$};
\node[left] at (s1) {$\bs_1$};
\draw [thick,decorate,decoration={brace,mirror,amplitude=3pt,raise=4pt}] (n1) -- (n2)
        node[midway,below=8pt]{$\bn_1$};
\draw [thick,decorate,decoration={brace,mirror,amplitude=3pt,raise=4pt}] (n3) -- (n4)
        node[midway,below=8pt]{$\bn_2$};

\draw [->,thick] ($(n4)+(.25,1)$) -- ++(0:1.5)
            node[midway,above]{$R_{\br_2,\bs_2}^{x,(s_{k_1},\bn_2)}$};
\end{scope}

\begin{scope}[xshift=110pt]
\node[Solid](x){}
    child{node[Solid](s2){}
        child{node[Solid](s1){}
            child[thin]{node[Solid](n1){}}
            child[thin]{node[Solid](n2){}}
            }
        child{
            child[missing]
            child{node[Solid](n3){}}
            }
        }
    child{
        child[missing]
        child{
            child[missing]
            child{node[Solid](n4){}}
            }
        };
\node[above] at (x) {$x$};
\node[left] at (s1) {$\bs_1$};
\node[left] at (s2) {$\bs_2$};
\draw [thick,decorate,decoration={brace,mirror,amplitude=3pt,raise=4pt}] (n1) -- (n2)
        node[midway,below=8pt]{$\bn_1$};
\draw [thick,decorate,decoration={brace,mirror,amplitude=3pt,raise=4pt}] (n3) -- (n4)
        node[midway,below=8pt]{$\bn_2$};
\end{scope}

\end{tikzpicture}
\]
The second identity follows from repeated application of the first identity.
\end{proof}
From \eqref{eq:R=J} it follows that $R_{\br,\bs}^{x,\bn}$ is essentially a multivariate $q$-Bessel function as defined by \eqref{eq:def multivariable Jnu}.
\begin{cor}
Let $\bnu(x,\bn) = (n_1,x+n_2,\ldots,x+n_{k+1},n_{k+2})$, then
\[
R_{\br,\bs}^{x,\bn} = (-q)^{r_1+s_k-n_1-n_{k+2}} J_{\bnu(x,\bn)}(\br,\bs;q^2).
\]
\end{cor}
Note that this corollary and Proposition \ref{prop:properties R multivariate} together give a representation theoretic proof of Theorem \ref{thm:properties multivariable J}\\

Our next goal is to prove a summation identity for the multivariate $q$-Bessel functions. Let us first mention that by interpreting a binary tree as a product of Clebsch-Gordan coefficients, the $3nj$-symbols $R_{\br,\bs}^{x,\bn}$ satisfy, by definition, the formula
\begin{equation} \label{eq:C=RC}
C_{x,\br,\bn} = \sum_{\bs} R_{\br,\bs}^{x,\bn} C_{x,\hat \bs,\hat \bn}, \qquad C_{x,\br,\bn} = \prod_{j=1}^{k+1} C_{r_{j-1},n_j,r_j},
\end{equation}
where $r_0=x, r_{k+1}=n_{k+2}, s_0=n_1, s_{k+2}=x$. The functions $C_{x,\br,\bn}$ can be considered as multivariate Wall polynomials, which are $q$-analogs of Laguerre polynomials. In this light \eqref{eq:C=RC} is a multivariate $q$-analog of an identity proved by Erd\'elyi \cite{Er} which states that the Hankel transform maps a product of two Laguerre polynomials to a product of two Laguerre polynomials.\\

For the $3nj$-symbols $R_{\br,\bs}^{x,\bn}$ there exists a multivariate analog of the Biedenharn-Elliott identity. In terms of $q$-Bessel functions this gives an expansion formula for $k$-variable $q$-Bessel functions in terms of $(k-1)$-variable $q$-Bessel functions. The identity requires also another $3nj$-symbol. For $\br,\bs \in \Z^k$, $\bn \in \Z^{k+2}$, $x \in \Z$, let $S_{\br,\bs}^{x,\bn}$ be the coupling coefficient defined by
\begin{equation} \label{eq:defS}
\begin{tikzpicture}
[baseline={([yshift=-.5ex]current bounding box.center)}]
\node[Solid](x){}
    child{
        child{node[Solid](n1){}}
        child[missing]
        }
    child{node[Solid](r){}
        child{node[Solid]{}}
        child{node[Solid](n3){}}
        };
\node[above] at (x) {$x$};
\node[left] at (r) {$\br$};
\draw [thick,decorate,decoration={brace,mirror,amplitude=5pt,raise=4pt}] (n1) -- (n3)
        node[midway,below=8pt]{$\bn$};
\end{tikzpicture}
= \sum_{\hat \bs} S_{\br,\bs}^{x,\bn} \
\begin{tikzpicture}
[baseline={([yshift=-.5ex]current bounding box.center)}]
\node[Solid](x){}
    child{node[Solid](s){}
        child{node[Solid](n1){}}
        child{node[Solid]{}}
        }
    child{
        child[missing]
        child{node[Solid](n3){}}
        };
\node[above] at (x) {$x$};
\node[left] at (s) {$\bs$};
\draw [thick,decorate,decoration={brace,mirror,amplitude=5pt,raise=4pt}] (n1) -- (n3)
        node[midway,below=8pt]{$\bn$};
\end{tikzpicture}
\end{equation}
Note that $\sum_{\hat s} = \sum_{s_1}\cdots \sum_{s_k}$. This $3nj$-symbol can of course also be considered as a multivariate $q$-Bessel function (see the following result), but it lacks the self-duality property.
Let us first express $S$ in terms of the $6j$-symbols.
\begin{lemma} \label{lem:S=product R}
$S_{\br,\bs}^{x,\bn}$ is given by
\[
S_{\br,\bs}^{x,\bn} = \prod_{j=1}^{k} R_{r_{j},s_{j}}^{s_{j+1},n_1,r_{j-1},n_{j+2}},
\]
with $s_{k+1}=x$ and $r_0=n_2$.
\end{lemma}
\begin{proof}
We use the transition
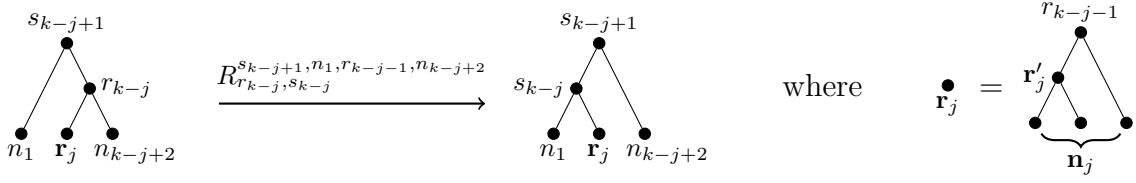
\[
\begin{tikzpicture}
[baseline={([yshift=-.5ex]current bounding box.center)}]
\node[Solid](x){}
    child{
        child{node[Solid](n1){}}
        child[missing]
        }
    child{node[Solid](r){}
        child{node[Solid](n2){}}
        child{node[Solid](n3){}}
        };
\node[above] at (x) {$s_{k-j+1}$};
\node[right] at (r) {$r_{k-j}$};

\node[below] at (n1) {$n_1$};
\node[below] at (n2) {$\br_{j}$};
\node[below,xshift=.3cm] at (n3) {$n_{k-j+2}$};

\draw [->,thick] (2,-.8) -- ++(0:3.5)
    node[draw=none,fill=none,rectangle,above,midway]
    {$R_{r_{k-j},s_{k-j}}^{s_{k-j+1},n_1,r_{k-j-1},n_{k-j+2}}$};

\node[Solid](x) at (7,0) {}
    child{node[Solid](s) {}
        child{node[Solid](n1){}}
        child{node[Solid](n2){}}
        }
    child{
        child[missing]
        child{node[Solid](n3){}}
        };
\node[above] at (x) {$s_{k-j+1}$};
\node[left] at (s) {$s_{k-j}$};

\node[below] at (n1) {$n_1$};
\node[below] at (n2) {$\br_{j}$};
\node[below,xshift=.3cm] at (n3) {$n_{k-j+2}$};

\end{tikzpicture}
\qquad \text{where} \qquad
\begin{tikzpicture}
[baseline={([yshift=.5ex]current bounding box.center)}]
\node[Solid] (r){};
\node[below] at (r) {$\br_{j}$};
\end{tikzpicture}
=
\begin{tikzpicture}
[baseline={([yshift=-.5ex]current bounding box.center)}]
\node[Solid](x){}
    child{node[Solid](r'){}
        child{node[Solid](n1){}}
        child{node[Solid]{}}
        }
    child{
        child[missing]
        child{node[Solid](n3){}}
        };
\node[above] at (x) {$r_{k-j-1}$};
\node[left] at (r') {$\br_{j}'$};
\draw [thick,decorate,decoration={brace,mirror,amplitude=5pt,raise=4pt}] (n1) -- (n3)
        node[midway,below=8pt]{$\bn_j$};
\end{tikzpicture}
\]
and where $\br_{j}'=(r_1,\ldots,r_{k-j-2})$ and $\bn_j= (n_2,\ldots,n_{k-j+1})$. We set $s_{k+1}=x$ and $r_0=n_2$, then applying this transition successively on subtrees for $j=0,\ldots,k-1$ gives
\[
S_{\br,\bs}^{x,\bn} = \prod_{j=0}^{k-1} R_{r_{k-j},s_{k-j}}^{s_{k-j+1},n_1,r_{k-j-1},n_{k-j+2}}.
\]
Changing the index gives the stated expression for the coupling coefficient $S$.
\end{proof}

The following identity is the multivariate analog of the Biedenharn-Elliott identity from Theorem \ref{thm:identitiesJ}, i.e., the $k=2$ case gives back the Theorem \ref{thm:identitiesJ}(ii).
\begin{thm} \label{thm:multivariate BHid}
For $k \in \N_{\geq 2}$ let $\br,\bs \in \Z^k$ and $\bn \in \Z^{k+2}$, then
\[
R_{\br,\bs}^{x,\bn} = \sum_{\bt \in \Z^{k-1}} S_{(\bt,r_1),\bs}^{x,\bn} R_{\br',\bt}^{r_1,\bn'},
\]
where $\bv'$ is obtained from $\bv$ by leaving out the first component. In terms of multivariate $q$-Bessel functions,
\[
J_{\bnu(x,\bn)}(\br,\bs) = \sum_{\bt \in \Z^{k-1}} A_{\bt,\bs}^{r_1} J_{\bnu(r_1,\bn')}(\br',\bt)
\]
with
\[
A_{\bt,\bs}^{r_1} = (-q^\frac12)^{|\bt|+|\bs|-|\bn|-(k-2)n_1-s_k+r_2} \prod_{j=1}^k J_{s_{j+1}-n_1+t_{j-1}+n_{j+2}}(q^{s_j+t_j-n_1-n_{j+2}};q), \qquad t_k=r_1.
\]
\end{thm}
\begin{proof}
This follows from the transition
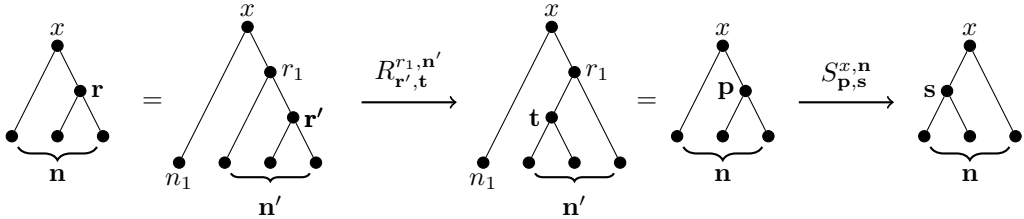
\[
\begin{tikzpicture}
[baseline={([yshift=-.5ex]current bounding box.center)}]
\node[Solid](x){}
    child{
        child{node[Solid](n1){}}
        child[missing]
        }
    child{node[Solid](r){}
        child{node[Solid]{}}
        child{node[Solid](n3){}}
        };
\node[above] at (x) {$x$};
\node[right] at (r) {$\br$};
\draw [thick,decorate,decoration={brace,mirror,amplitude=5pt,raise=4pt}] (n1) -- (n3)
        node[midway,below=8pt]{$\bn$};


\node at (1.25,-.75) {$=$};


\node[Solid](x) at (2.5,.25){}
    child{
        child{
            child{node[Solid](n1){}}
            child[missing]
            }
        child[missing]
        }
    child{node[Solid](r1){}
        child{
            child{node[Solid](n2){}}
            child[missing]
        }
        child{node[Solid](r2){}
            child[thin]{node[Solid](n3){}}
            child[thin]{node[Solid](n4){}}
            }
        };
\node[above] at (x) {$x$};
\node[right] at (r2) {$\br'$};
\node[right] at (r1) {$r_1$};
\node[below] at (n1) {$n_1$};
\draw [thick,decorate,decoration={brace,mirror,amplitude=3pt,raise=4pt}] (n2) -- (n4)
        node[midway,below=8pt]{$\bn'$};


\draw [->,thick] (4,-.75) -- ++(0:1.25)
    node[draw=none,fill=none,rectangle,above,midway]{$R_{\br',\bt}^{r_1,\bn'}$};


\node[Solid](x) at (6.5,.25){}
    child{
        child{
            child{node[Solid](n1){}}
            child[missing]
            }
        child[missing]
        }
    child{node[Solid](r1){}
        child{node[Solid](p){}
            child{node[Solid](n2){}}
            child{node[Solid]{}}
            }
        child{
            child[missing]
            child[thin]{node[Solid](n4){}}
            }
        };
\node[above] at (x) {$x$};
\node[right] at (r2) {$\br'$};
\node[right] at (r1) {$r_1$};
\node[left] at (p) {$\bt$};
\node[below] at (n1) {$n_1$};
\draw [thick,decorate,decoration={brace,mirror,amplitude=3pt,raise=4pt}] (n2) -- (n4)
        node[midway,below=8pt]{$\bn'$};


\node at (7.75,-.75) {$=$};


\node[Solid](x) at (8.75,0){}
    child{
        child{node[Solid](n1){}}
        child[missing]
        }
    child{node[Solid](p){}
        child{node[Solid]{}}
        child{node[Solid](n3){}}
        };
\node[above] at (x) {$x$};
\node[left] at (p) {$\bp$};
\draw [thick,decorate,decoration={brace,mirror,amplitude=5pt,raise=4pt}] (n1) -- (n3)
        node[midway,below=8pt]{$\bn$};


\draw [->,thick] (9.75,-.75) -- ++(0:1.25)
    node[draw=none,fill=none,rectangle,above,midway]{$S_{\bp,\bs}^{x,\bn}$};


\node[Solid](x) at (12,0){}
    child{node[Solid](s){}
        child{node[Solid](n1){}}
        child{node[Solid]{}}
        }
    child{
        child[missing]
        child{node[Solid](n3){}}
        };
\node[above] at (x) {$x$};
\node[left] at (s) {$\bs$};
\draw [thick,decorate,decoration={brace,mirror,amplitude=5pt,raise=4pt}] (n1) -- (n3)
        node[midway,below=8pt]{$\bn$};

\end{tikzpicture}
\]
where $\bp= (\bt,r_1)$, and the definition of the coupling coefficients $R$.
\end{proof}

\begin{remark}
It seems that there are no analogs for the $3nj$-symbols $R$ of identities (i) and (iii) of Theorem \ref{thm:identitiesJ}, but there does exist an analog of Theorem \ref{thm:identitiesJ}(i) involving only the $3nj$-symbols $S$ which may be of interest. This is obtained as follows.

Let $\bn \in \Z^{k+2}$. For $j \in \{1,2,\ldots,k+1\}$ we define $\bn_j = (n_{k+3-j},\ldots,n_{k+2},n_1,\ldots,n_{k+2-j})$. Furthermore, given a vector $\bv$, we denote (as in Theorem \ref{thm:multivariate BHid}) by $\bv'$ the vector $\bv$ without the first component, and we set $\bn_j'= (\bn_j)'$. Consider the transition
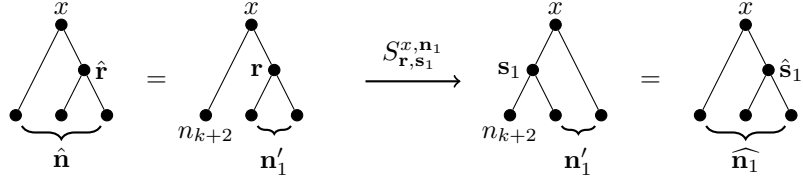
\[
\begin{tikzpicture}
[baseline={([yshift=-.5ex]current bounding box.center)}]
\node[Solid](x){}
    child{
        child{node[Solid](n1){}}
        child[missing]
        }
    child{node[Solid](r){}
        child{node[Solid](n2){}}
        child{node[Solid](n3){}}
        };
\node[above] at (x) {$x$};
\node[right] at (r) {$\hat \br$};
\draw [thick,decorate,decoration={brace,mirror,amplitude=5pt,raise=4pt}] (n1) -- (n3)
        node[midway,below=8pt]{$\hat \bn$};


\node at (1.25,-.75) {$=$};


\node[Solid](x) at (2.5,0){}
    child{
        child{node[Solid](n1){}}
        child[missing]
        }
    child{node[Solid](r){}
        child{node[Solid](n2){}}
        child{node[Solid](n3){}}
        };
\node[above] at (x) {$x$};
\node[left] at (r) {$\br$};
\node[below] at (n1) {$n_{k+2}$};
\draw [thick,decorate,decoration={brace,mirror,amplitude=3pt,raise=4pt}] (n2) -- (n3)
        node[midway,below=8pt]{$\bn_1'$};


\draw [->,thick] (4,-.75) -- ++(0:1.25)
    node[draw=none,fill=none,rectangle,above,midway]{$S_{\br,\bs_1}^{x,\bn_1}$};

\node[Solid](x) at (6.5,0){}
    child{node[Solid](s){}
        child{node[Solid](n1){}}
        child{node[Solid](n2){}}
        }
    child{
        child[missing]
        child{node[Solid](n3){}}
        };
\node[above] at (x) {$x$};
\node[left] at (s) {$\bs_1$};
\node[below] at (n1) {$n_{k+2}$};
\draw [thick,decorate,decoration={brace,mirror,amplitude=3pt,raise=4pt}] (n2) -- (n3)
        node[midway,below=8pt]{$\bn_1'$};


\node at (7.75,-.75) {$=$};


\node[Solid](x) at (9,0){}
    child{
        child{node[Solid](n1){}}
        child[missing]
        }
    child{node[Solid](r){}
        child{node[Solid](n2){}}
        child{node[Solid](n3){}}
        };
\node[above] at (x) {$x$};
\node[right] at (r) {$\hat \bs_1$};
\draw [thick,decorate,decoration={brace,mirror,amplitude=5pt,raise=4pt}] (n1) -- (n3)
        node[midway,below=8pt]{$\widehat {\bn_1}$};
\end{tikzpicture}
\]
Iterating this transition $k+1$ times shows that the coupling coefficient in the transition
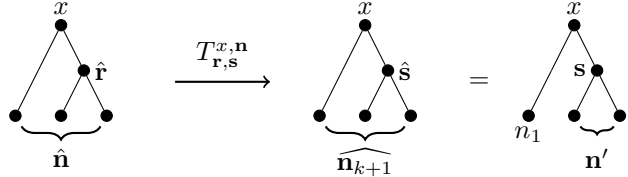
\[
\begin{tikzpicture}
[baseline={([yshift=-.5ex]current bounding box.center)}]
\node[Solid](x){}
    child{
        child{node[Solid](n1){}}
        child[missing]
        }
    child{node[Solid](r){}
        child{node[Solid](n2){}}
        child{node[Solid](n3){}}
        };
\node[above] at (x) {$x$};
\node[right] at (r) {$\hat \br$};
\draw [thick,decorate,decoration={brace,mirror,amplitude=5pt,raise=4pt}] (n1) -- (n3)
        node[midway,below=8pt]{$\hat \bn$};


\draw [->,thick] (1.5,-.75) -- ++(0:1.25)
    node[draw=none,fill=none,rectangle,above,midway]{$T_{\br,\bs}^{x,\bn}$};


\node[Solid](x) at (4,0){}
    child{
        child{node[Solid](n1){}}
        child[missing]
        }
    child{node[Solid](r){}
        child{node[Solid](n2){}}
        child{node[Solid](n3){}}
        };
\node[above] at (x) {$x$};
\node[right] at (r) {$\hat \bs$};
\draw [thick,decorate,decoration={brace,mirror,amplitude=5pt,raise=4pt}] (n1) -- (n3)
        node[midway,below=8pt]{$\widehat {\bn_{k+1}}$};


\node at (5.5,-.75) {$=$};


\node[Solid](x) at (6.75,0){}
    child{
        child{node[Solid](n1){}}
        child[missing]
        }
    child{node[Solid](r){}
        child{node[Solid](n2){}}
        child{node[Solid](n3){}}
        };
\node[above] at (x) {$x$};
\node[left] at (r) {$\bs$};
\node[below] at (n1) {$n_{1}$};
\draw [thick,decorate,decoration={brace,mirror,amplitude=3pt,raise=4pt}] (n2) -- (n3)
        node[midway,below=8pt]{$\bn'$};
\end{tikzpicture}
\]
is given by
\[
T_{\br,\bs}^{x,\bn} = \sum_{\bs_{k} } \cdots \sum_{\bs_1} \left(S_{\bs_0,\bs_1}^{x,\bn_1} \cdots S_{\bs_{k},\bs_{k+1}}^{x,\bn_{k+1}}\right), \qquad \bs_0 = \br, \quad \bs_{k+1} = \bs.
\]
On the other hand, by the definition of the coupling coefficient $S$ we have
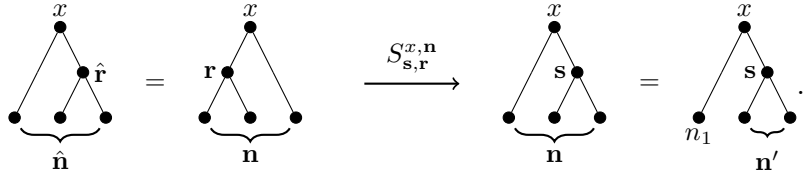
\[
\begin{tikzpicture}
[baseline={([yshift=-.5ex]current bounding box.center)}]
\node[Solid](x){}
    child{
        child{node[Solid](n1){}}
        child[missing]
        }
    child{node[Solid](r){}
        child{node[Solid](n2){}}
        child{node[Solid](n3){}}
        };
\node[above] at (x) {$x$};
\node[right] at (r) {$\hat \br$};
\draw [thick,decorate,decoration={brace,mirror,amplitude=5pt,raise=4pt}] (n1) -- (n3)
        node[midway,below=8pt]{$\hat \bn$};


\node at (1.25,-.75) {$=$};

\node[Solid](x) at (2.5,0){}
    child{node[Solid](s){}
        child{node[Solid](n1){}}
        child{node[Solid](n2){}}
        }
    child{
        child[missing]
        child{node[Solid](n3){}}
        };
\node[above] at (x) {$x$};
\node[left] at (s) {$\br$};
\draw [thick,decorate,decoration={brace,mirror,amplitude=5pt,raise=4pt}] (n1) -- (n3)
        node[midway,below=8pt]{$\bn$};


\draw [->,thick] (4,-.75) -- ++(0:1.25)
    node[draw=none,fill=none,rectangle,above,midway]{$S_{\bs,\br}^{x,\bn}$};


\node[Solid](x) at (6.5,0){}
    child{
        child{node[Solid](n1){}}
        child[missing]
        }
    child{node[Solid](r){}
        child{node[Solid](n2){}}
        child{node[Solid](n3){}}
        };
\node[above] at (x) {$x$};
\node[left] at (r) {$\bs$};
\draw [thick,decorate,decoration={brace,mirror,amplitude=5pt,raise=4pt}] (n1) -- (n3)
        node[midway,below=8pt]{$\bn$};


\node at (7.75,-.75) {$=$};


\node[Solid](x) at (9,0){}
    child{
        child{node[Solid](n1){}}
        child[missing]
        }
    child{node[Solid](r){}
        child{node[Solid](n2){}}
        child{node[Solid](n3){}}
        };
\node[above] at (x) {$x$};
\node[left] at (r) {$\bs$};
\node[below] at (n1) {$n_{1}$};
\draw [thick,decorate,decoration={brace,mirror,amplitude=3pt,raise=4pt}] (n2) -- (n3)
        node[midway,below=8pt]{$\bn'$};
\end{tikzpicture}.
\]
so that
\[
S_{\bs,\br}^{x,\bn} = \sum_{\bs_{k} } \cdots \sum_{\bs_1} \left(S_{\bs_0,\bs_1}^{x,\bn_1} \cdots S_{\bs_{k},\bs_{k+1}}^{x,\bn_{k+1}}\right), \qquad \bs_0 = \br, \quad \bs_{k+1} = \bs.
\]
For $k=1$ this gives back Theorem \ref{thm:identitiesJ}(i).
\end{remark}


\begin{thebibliography}{99}


\bibitem{CFS} J.S.~Carter, D.E.~Flath, M.~Saito, \textit{The classical and quantum 6j-symbols}, Mathematical Notes, \textbf{43}. Princeton University Press, Princeton, NJ, 1995.

\bibitem {Er} A.~Erd\'elyi, \textit{The Hankel Transform of a Product of Whittaker's Functions}, J.~London Math.~Soc.~\textbf{13}, no. 2 (1938), 146--154.

\bibitem{GR} G.~Gasper, M.~Rahman, \textit{Basic Hypergeometric Series}, 2nd ed., Cambridge University Press, Cambridge, 2004.

\bibitem{GR05} G.~Gasper, M.~Rahman, \textit{Some systems of multivariable orthogonal Askey-Wilson polynomials}, Theory and applications of special functions, 209--219, Dev.~Math., \textbf{13}, Springer, New York, 2005.

\bibitem{GR07} G.~Gasper, M.~Rahman, \textit{Some systems of multivariable orthogonal q-Racah polynomials}, Ramanujan J.~\textbf{13} (2007), no.~ 1-3, 389--405.

\bibitem{GeIl} J.S.~Geronimo, P.~Iliev, \textit{Multivariable Askey-Wilson function and bispectrality}, Ramanujan J.~\textbf{24} (2011), no.~3, 273--287.

\bibitem{Groen14}  W.~Groenevelt, \textit{Coupling coefficients for tensor product representations of quantum SU(2)}, J.~Math.~Phys.~\textbf{55} (2014).

\bibitem{Il} P.~Iliev, \textit{Bispectral commuting difference operators for multivariable Askey-Wilson polynomials}, Trans.~Amer.~Math.~Soc.~\textbf{363} (2011), no.~3, 1577--1598.

\bibitem{Jeugt} J. Van der Jeugt, \textit{$3nj$-coefficients and orthogonal polynomials of hypergeometric type}, Orthogonal polynomials and special functions (Leuven, 2002), 25–92, Lecture Notes in Math., \textbf{1817}, Springer, Berlin, 2003.


\bibitem{Ji} M.~Jimbo, \textit{Introduction to the Yang-Baxter equation}, Internat.~J.~Modern Phys.~A \textbf{4} (1989), no.~15, 3759--3777.

\bibitem{KalMilMu} E.G.~Kalnins, W.~Miller Jr., S.~Mukherjee, \textit{Models of $q$-algebra representations: the group of plane motions}, SIAM J.~Math.~Anal.~\textbf{25} (1994), no.~2, 513--527.

\bibitem{KLS}  R.~Koekoek, P.A.~Lesky, R.~Swarttouw, \textit{Hypergeometric orthogonal polynomials and their q-analogues}, Springer Monographs in Mathematics. Springer-Verlag, Berlin, 2010.

\bibitem{KoeSt} E.~Koelink, J.V.~Stokman, \textit{The Askey-Wilson function transform scheme}, Special functions 2000: current perspective and future directions (Tempe, AZ), 221--241, NATO Sci. Ser. II Math. Phys. Chem., \textbf{30}, Kluwer Acad. Publ., Dordrecht, 2001.

\bibitem{Koe94} H.T.~Koelink, \textit{The quantum group of plane motions and the Hahn-Exton $q$-Bessel function}, Duke Math.~J.~\textbf{76} (1994), no.~2, 483--508.

\bibitem{KoeSw}  H.T.~Koelink, R.F.~Swarttouw, \textit{A $q$-analogue of Graf's addition formula for the Hahn-Exton $q$-Bessel function}, J.~Approx.~Theory \textbf{81} (1995), no.~2, 260--273.

\bibitem{KooSw}  T.H.~Koornwinder, R.F.~Swarttouw, \textit{On $q$-analogues of the Fourier and Hankel transforms}, Trans.~Amer.~Math.~Soc.~\textbf{333} (1992), no.~1, 445--461. See 	\textsf{arXiv:1208.2521 [math.CA]} for a corrected version.

\bibitem{Ro} H.~Rosengren, \textit{Multivariable q-Hahn polynomials as coupling coefficients for quantum algebra representations}, Int.~J.~Math.~Math.~Sci.~\textbf{28} (2001), no.~6, 331--358.

\bibitem{Sc} F.~Scarabotti, \textit{The tree method for multidimensional q-Hahn and q-Racah polynomials}, Ramanujan J.~\textbf{25} (2011), no.~1, 57--91.

\bibitem{Tr} M.V.~Tratnik, \textit{Multivariable Wilson polynomials}, J.~Math.~Phys.~\textbf{30} (1989), no.~9, 2001--2011.


\end{thebibliography}
\end{document}